\documentclass[a4paper,11pt,reqno]{article}

\usepackage[utf8]{inputenc}
\usepackage[T1]{fontenc}
\usepackage{lmodern}
\usepackage[english]{babel}
\usepackage{microtype}
\usepackage{ulem}
\usepackage{amsmath,amssymb,amsfonts,amsthm}
\usepackage{mathtools,accents}
\usepackage{mathrsfs}
\usepackage{aliascnt}
\usepackage{braket}
\usepackage{bm}
\usepackage{cancel}
\usepackage[a4paper,margin=3cm]{geometry}
\usepackage[citecolor=blue,colorlinks]{hyperref}

\usepackage{enumerate}
\usepackage{xcolor}

\makeatletter
\g@addto@macro\@floatboxreset\centering
\makeatother

\makeatletter
\def\newaliasedtheorem#1[#2]#3{
  \newaliascnt{#1@alt}{#2}
  \newtheorem{#1}[#1@alt]{#3}
  \expandafter\newcommand\csname #1@altname\endcsname{#3}
}
\makeatother

\numberwithin{equation}{section}

\newtheoremstyle{slanted}{\topsep}{\topsep}{\slshape}{}{\bfseries}{.}{.5em}{}

\theoremstyle{plain}
\newtheorem{theorem}{Theorem}[section]
\newaliasedtheorem{proposition}[theorem]{Proposition}
\newaliasedtheorem{lemma}[theorem]{Lemma}
\newaliasedtheorem{corollary}[theorem]{Corollary}
\newaliasedtheorem{conjecture}[theorem]{Conjecture}
\newaliasedtheorem{counterexample}[theorem]{Counterexample}

\theoremstyle{definition}
\newaliasedtheorem{definition}[theorem]{Definition}
\newaliasedtheorem{question}[theorem]{Question}
\newaliasedtheorem{example}[theorem]{Example}

\theoremstyle{remark}
\newaliasedtheorem{remark}[theorem]{Remark}

\newcommand{\setR}{\mathbb{R}}

\newcommand{\eps}{\varepsilon}

\let\altphi\phi
\let\phi\varphi
\let\varphi\altphi
\let\altphi\undefined

\DeclareMathOperator{\supp}{supp}

\newcommand{\Ch}{{\sf Ch}}

\DeclareMathOperator{\Lip}{Lip}

\newcommand{\haus}{\mathscr{H}}

\newcommand{\dist}{\mathsf{d}}

\newcommand{\meas}{\mathfrak{m}}
\newcommand{\bist}{\mathsf{b}}

\DeclareMathOperator{\RCD}{RCD}
\DeclareMathOperator{\CD}{CD}

\newfont{\tmpf}{cmsy10 scaled 2500}

\newcommand{\intav}{{\mathop{\int\kern-10pt\rotatebox{0}{\textbf{--}}}}}
\renewcommand{\d}{\mathrm{d}}
\renewcommand{\ }{\text{ }}
\def\<{\langle}
\def\>{\rangle}
\newcommand{\ra}{\rightarrow}
\newcommand{\R}{\setR}

\def\Prod#1#2{\mathop{\prod}\limits_{#1}^{#2}}
\newcommand{\Vol}{\mathop{\mathrm{Vol}}}
\def\esssup#1{\mathop{\mathrm{ess\text{ }sup}}\limits_{#1}}
\def\essinf#1{\mathop{\mathrm{ess\text{ }inf}}\limits_{#1}}

\newcommand{\lip}{\mathop{\mathrm{lip}}}

\newcommand{\h}{\mathrm{h}}

\begin{document}
\title{Sharp gradient estimate, rigidity and almost rigidity of Green functions on non-parabolic $\RCD(0,N)$ spaces}
 
\author{
Shouhei Honda
\thanks{Mathematical Institute, Tohoku University; \url{shouhei.honda.e4@tohoku.ac.jp}}
\,and 	Yuanlin Peng
\thanks{Mathematical Institute, Tohoku University; JSPS Research Fellow; \url{peng.yuanlin.p6@dc.tohoku.ac.jp }} } 
\maketitle
\begin{abstract}
	Inspired by a result in \cite{C12} of Colding, 
	 the present paper studies the Green function $G$ on a non-parabolic $\RCD(0,N)$ space $(X, \dist, \meas)$ for some finite $N>2$. Defining $\bist_x=G(x, \cdot)^{\frac{1}{2-N}}$ for a point $x \in X$, which plays a role of a smoothed distance function from $x$,  
	 we prove that the gradient $|\nabla \bist_x|$ has the canonical pointwise representative with the sharp upper bound in terms of the $N$-volume density $\nu_x=\lim_{r\to 0^+}\frac{\meas (B_r(x))}{r^N}$ of $\meas$ at $x$;
	 \begin{equation*}
	 |\nabla \bist_x|(y) \le \left(N(N-2)\nu_x\right)^{\frac{1}{N-2}}, \quad \text{for any $y \in X \setminus \{x\}$}.
	 \end{equation*}
	 Moreover the rigidity is obtained, namely, the upper bound is attained at a point $y \in X \setminus \{x\}$ 
	 if and only if the space is isomorphic to the $N$-metric measure cone over an $\RCD(N-2, N-1)$ space. 
	 In the case when $x$ is an $N$-regular point, the rigidity states an isomorphism to the $N$-dimensional Euclidean space $\mathbb{R}^N$, thus, this extends the result of Colding 
	  to $\RCD(0,N)$ spaces. It is emphasized that the almost rigidities  are also proved, which are new even in the smooth framework. 
	\end{abstract}
\tableofcontents
\section{Introduction}
\subsection{Green function and concerned problems in the smooth framework}
In the classical PDE theory, the (positive) Green function $G_x$ at the pole $x$ of the Laplace operator on the $N$-dimensional Euclidean space $\R^N$ is the solution to the heat equation
\begin{equation}
\Delta u=-\delta_x
\end{equation}
as measures, where $\delta_x$ is the Dirac measure at $x$. In the case when $N\geqslant 2$, it is well-known that this equation is solved by
\begin{equation}\label{eq:01}
G_x=\left\{
\begin{aligned}
&-\frac1{2\pi}\log\dist_x,&N=2\\
&\frac{\dist_x^{2-N}}{N(N-2)\omega_N},&N\geqslant 3
\end{aligned}\right.
\end{equation}
where $\dist_x(\cdot)$ is the Euclidean distance function from $x$ and $\omega_N:=\pi^{\frac{N}{2}}\Gamma (\frac{N}{2}+1)^{-1}$ is the volume of a unit ball $B_1(0_N)$ in $\mathbb{R}^N$.

We can also discuss the Green functions for more general classes of spaces along the same line. For instance, it is known that for an $N$-dimensional complete Riemannian manifold $(M^N,g)$ ($N\geqslant 2$) with non-negative Ricci curvature, the existence of the (global) Green function $G$ is equivalent to the following non-parabolic assumption:
\begin{equation}
\int_1^\infty\frac{r}{\Vol B_r(x)}\ \d r<\infty, \quad \forall x\in M^N,
\end{equation}
where $B_r(x)$ denotes the open ball centered at $x$ of radius $r$ with respect to the induced distance $\dist$ by $g$, and $\Vol$ denotes the Riemannian volume measure by $g$.
See \cite{Varopoulos} by Varapoulos for the details. In this case, it is well-known that the following asymptotic behavior for the Green function $G_x$ at the pole $x \in M^N$ holds as $\dist_x\ra 0^+$:
\begin{equation}\label{eq:03}
G_x=\left\{
\begin{aligned}
&-\frac 1{2\pi}\log \dist_x+o(-\log\dist_x), &N=2\\
&\frac{\dist_x^{2-N}}{N(N-2)\omega_N}+o(\dist_x^{2-N}), &N\geqslant 3
\end{aligned}
\right.
\end{equation}
This fact indicates that in the case when $N\geqslant 3$, the function
\begin{equation}\label{def bist}
\bist_x:=G_x^\frac 1{2-N}
\end{equation}
should be a counterpart of the distance function from $x$ up to a multiplication of a dimensional positive constant;
\begin{equation}
(N(N-2)\omega_N)^\frac{1}{N-2} \dist_x. 
\end{equation}
Colding \cite{C12} proved the sharp gradient estimate for $\bist_x$ and the rigidity as follows.
\begin{theorem}[Theorem 3.1 of \cite{C12}]\label{thm:C12}
	Let  $(M^N,g)$ be an $N$-dimensional ($N \geqslant 3$), complete and non-parabolic Riemannian manifold  with non-negative Ricci curvature and let $x\in M^N$. Then we have the following. 
	\begin{enumerate}
		\item{(Sharp gradient estimate)} We have 
		\begin{equation}\label{sharp constant}
		|\nabla \bist_x|(z)\leqslant  (N(N-2)\omega_N)^\frac{1}{N-2}
		\end{equation} for any $z \in M^N \setminus \{x\}$.
		\item{(Rigidity)} $(M^N,g)$ is isometric to the $N$-dimensional Euclidean space $\mathbb{R}^N$ with $\bist_x=(N(N-2)\omega_N)^\frac{1}{N-2}\dist_x$ if the equality of (\ref{sharp constant}) holds for some $z \in M^N \setminus \{x\}$.
	\end{enumerate}
	\end{theorem}
	Note that Colding used the normalized one, $(N(N-2)\omega_N)^\frac{1}{2-N}G_x^\frac 1{2-N}$, as the definition of $\bist_x$. Thus the sharp upper bound in \cite{C12} was exactly $1$ instead of the right-hand-side of (\ref{sharp constant}). See the footnote $4$ in \cite{C12}.
	
In particular the rigidity indicates that $\bist_x$ exactly coincides with $(N(N-2)\omega_N)^\frac{1}{N-2} \dist_x$ if and only if the manifold is Euclidean. 
Given this rigidity result, it is natural to ask whether the quantitative almost rigidity result is satisfied or not:
\begin{itemize}
	\item[(Q)]\label{Q} If $|\nabla \bist_x|(y)$ is close to the sharp upper bound $(N(N-2)\omega_N)^\frac{1}{N-2}$ at some point $y \in M^N \setminus \{x\}$, then can we conclude that the manifold is pointed Gromov-Hausdorff (pGH) close to $\mathbb{R}^N$? 
\end{itemize}
It is worth mentioning that
\begin{equation}\label{as shar}
|\nabla \bist_x|(y) \to  (N(N-2)\omega_N)^\frac{1}{N-2}
\end{equation}
whenever $y \to x$. Therefore in order to give a positive answer to the question (Q), we need to find an additional assumption on $y$.

We are now in a position to introduce the first main result of the paper. 
\begin{theorem}[Almost rigidity]\label{thm:smooth almost}
For any integer $N \geqslant 3$, all $0<\eps<1$, $0<r<R$, $1 \leqslant p <\infty$ and $\phi \in L^1([0,\infty), \haus^1)$ there exists $\delta:=\delta(N, \eps, r, R, p, \phi)>0$ such that if 
an $N$-dimensional ($N \geqslant 3$) complete Riemannian manifold with non-negative Ricci curvature $(M^N, g)$ satisfies
\begin{equation}
\frac{s}{\Vol B_s(x)} \leqslant \phi(s),\quad \text{for $\haus^1$-a.e. $s \in [1, \infty)$}
\end{equation}
for some $x \in M^N$ and that
\begin{equation}\label{almost max point}
(N(N-2)\omega_N)^\frac{1}{N-2}-|\nabla \bist_x|(y) \leqslant \delta 
\end{equation}
holds for some $y \in \overline{B}_R(x) \setminus B_r(x)$. Then we have 
\begin{align}\label{r lower}
\dist_{\mathrm{pmGH}}\left( (M^N, \dist, \Vol, x), (\mathbb{R}^N, \dist_{\mathbb{R}^N}, \haus^N, 0_N) \right)<\eps
\end{align}
and 
\begin{align}\label{impsosirbb}
&\left\| \bist_x-(N(N-2)\omega_N)^{\frac{1}{N-2}}\dist_x\right\|_{L^{\infty}(B_R(x))} \nonumber \\
&+ \left\| \bist_x-(N(N-2)\omega_N)^{\frac{1}{N-2}}\dist_x\right\|_{H^{1, p}(B_R(x), \dist, \Vol /\Vol B_R(x))}\leqslant \eps
\end{align}
in particular
\begin{equation}\label{elpest222}
\|(N(N-2)\omega_N)^\frac{1}{N-2}-|\nabla \bist_x| \|_{L^p(B_R(x), \Vol /\Vol B_R(x))} \leqslant \eps
\end{equation}
where $\dist_{\mathrm{pmGH}}$ denotes any fixed distance metrizing the pointed measured Gromov-Hausdorff (pmGH) convergence. 
\end{theorem}
As explained around (\ref{as shar}),  the lower bound $r$ in Theorem \ref{almost max point} cannot be dropped in order to get (\ref{r lower}).
On the other hand, it is known that if the asymptotic $N$-volume $V_{M^N}$ defined by
\begin{equation}
V_{M^N}:=\lim_{R \to \infty}\frac{\Vol B_R(x)}{R^N} (\leqslant \omega_N)
\end{equation}
is close to $\omega_N$, then $(M^N, \dist, \Vol, x)$ is pmGH close to $(\mathbb{R}^N, \dist_{\mathbb{R}^N}, \haus^N, 0_N)$, quantitatively. See \cite{C97} by Colding. Note that the converse statement is not true even in the case when the metric is Ricci flat with the maximal volume growth (namely $V_{M^N}>0$). 

In connection with this observation, it is natrual to ask whether the conclusion (\ref{r lower}) in the theorem above can be improved to be that $V_{M^N}$ is close to $\omega_N$, or not. However a simple blow-up argument on a fixed manifold which is not isometric to $\mathbb{R}^N$ allows us to conclude that the desired improvement is impossible, see also Remark \ref{rem:EH}. 

As another possible improvement in the theorem above, it is also natural to ask whether the case when $p=\infty$ in (\ref{elpest222}) is satisfied or not, namely
\begin{equation}\label{elinftyest}
(N(N-2)\omega_N)^\frac{1}{N-2}-|\nabla \bist_x| \leqslant \eps,\quad \text{on $B_R(x)$?}
\end{equation}
However we can also see that this improvement  is impossible (thus the improvement of (\ref{impsosirbb}) to the case when $p=\infty$ is also impossible) via Gromov-Hausdorff limits. See subsection \ref{sharp obs}.


The obsevation above allows us to say that Theorem \ref{thm:smooth almost} is sharp.
Finally let us introduce an immediate corollary.
\begin{corollary}\label{cor main}
For any integer $N \geqslant 3$, all $0<\eps<1$ and $v>0$ there exists $\delta:=\delta (N, \eps, v)>0$ such that
if an $N$-dimensional ($N \geqslant 3$) complete Riemannian manifold with non-negative Ricci curvature $(M^N, g)$ satisfies
$V_{M^N} \geqslant v$ and (\ref{almost max point}) for some sequence $y_i \in M^N(i=1,2,\ldots)$ with $\dist(x, y_i)\to \infty$,
then 
\begin{equation}
\left| V_{M^N}- \omega_N \right| \leqslant \eps
\end{equation}
In particular, in addition, if $\eps$ is sufficiently small depending only on $N$ and $v$, then $M^N$ is diffeomorphic to $\mathbb{R}^N$.
\end{corollary}
Note that the existence of such sequence $y_i$ in the corollary above cannot be replaced by the existence of only one point which is far from $x$.  See Remark \ref{final rem}.

The results above are justified via a non-smooth geometric analysis with Ricci curvature bounded below. Moreover the results above are generalized to such a non-smooth framework, so-called $\RCD$ spaces.
In the next section let us provide a brief introduction on $\RCD$ spaces.
\subsection{Non-smooth space with Ricci curvature bounded below; $\RCD$ spaces}
In the first decade of this century, Lott-Villani \cite{LottVillani} and Sturm \cite{Sturm06,Sturm06b} introduced the notion of $\CD(K,N)$ spaces independently as a concept of metric measure spaces with Ricci curvature bounded below by $K\in\R$ and dimension bounded above by $N\in[1,\infty]$ in some synthetic sense via the optimal transportation theory. For instance, in the case when $N$ is an integer,  $\mathbb{R}^N$ with any norm and the Lebesgue measure $\mathscr L^N$ satisfies the $\CD(0, N)$ condition.  Note that this is not ``Riemannian'' whenever the norm does not come from an inner product and that Gigli found a ``Riemannian'' notion on general metric measure spaces, so-called infinitesimally Hilbertianity,  in \cite{Gigli13}, which allows us to meet the Dirichlet form theory from the  metric measure geometry. It is worth mentioning that $N$ is not necessarily to be an integer in general.

{
After a pioneer work of Gigli-Kuwada-Ohta \cite{GigliKuwadaOhta} on Alexandrov spaces, Ambrosio-Gigli-Savar\'e (in the case when $N=\infty$) and Gigli (in the case when $N<\infty$) introduced $\RCD(K, N)$ spaces (or $\RCD$ spaces for short) by adding the infinitesimally Hilbertianity to the $\CD$ condition. It is known that $\RCD$ spaces include 	
weighted Riemannian manifolds with Bakry-\'Emery Ricci curvature bounded below, Ricci limit spaces, and Alexandrov spaces \cite{Petrunin, ZhangZhu} by Petrunin and Zhang-Zhu. The study is hugely developed, see for instance \cite{A, GIGLI} as nice surveys.

	

As explained in the previous subsection, we will mainly discuss an $\RCD(0, N)$ space  $(X,\dist,\meas)$ for some finite $N >2$ satisfying the non-parabolic assumption:
\begin{equation}\label{eq:a}
\int_1^\infty \frac{r}{\meas(B_r(x))}\ \d\meas<\infty, \quad \forall x\in X.
\end{equation}
Then, as in the smooth case, the global Green function $G=G^X$ can be defined by the integration of the heat kernel $p(x,y,t)$:
\begin{equation}
G(x,y):=\int_0^\infty p(x,y,t)\ \d t,
\end{equation}
and it is proved in \cite{BrueSemola} by Bru\`e-Semola that $G$ is well-defined with $G(x,\cdot)\in W^{1,1}_{\mathrm{loc}}(X,\dist,\meas)$ for any $x\in X$. 
A typical example of $\RCD(0, N)$ spaces is
\begin{equation}\label{metric meas}
\left([0,\infty), \dist_{\mathrm{Euc}}, r^{N-1}dr\right)
\end{equation}
whose Green function $G$ satisfies the following expression from the pole/origin $0$;
\begin{equation}\label{s7shsnsbbsys6}
G(0, r)=\frac{1}{N-2}r^{2-N}.
\end{equation}
See Proposition \ref{explicit green}. It is worth mentioning that (\ref{metric meas}) is the $N$-metric measure cone over a single point (Definition \ref{def:cone}).
\subsection{Main results and organization of this paper}
In order to introduce main results of this paper, fix an $\RCD(0,N)$ space $(X,\dist,\meas)$ for some finite $N >2$ satisfying the non-parabolic assumption \eqref{eq:a}. Moreover we also fix a point $x\in X$ whose $N$-volume density $\nu_x$ is finite;
\begin{equation}\label{eq:67}
\nu_x:=\lim_{r\ra 0}\frac{\meas(B_r(x))}{r^N}\in(0,\infty),
\end{equation}
where the positivity is a direct consequence of the Bishop-Gromov inequality. 
\begin{remark}
The origin of the $\RCD(0, N)$ space (\ref{metric meas}) satisfies (\ref{eq:67}), more generaly (\ref{eq:67}) is satisfied at the pole of any $N$-metric measure cone over an $\RCD(N-2, N-1)$ space. 
It is worth mentioning that (\ref{eq:67}) is also satisfied at any point if the space is non-collapsed, namely $\meas=\haus^N$ because of the Bishop inequality, where $\haus^N$ denotes the $N$-dimensional Hausdorff measure (see Definition \ref{ncrcddef}).
\end{remark}

Defining $\bist_x$ by (\ref{def bist}) in this setting, let us ask whether a similar rigidity result as in Theorem \ref{thm:C12} is justified even in this setting, or not. The main difficulty to realize this consist of two parts;
\begin{itemize}
\item a priori, $|\nabla \bist_x|$ makes only $\meas$-a.e. sense;
\item as observed in (\ref{s7shsnsbbsys6}), in general, the constancy of $|\nabla \bist_x|$ does not imply an isomorphism to a Euclidean space.
\end{itemize}

The first main result in this setting are stated as follows, which allow us to overcome the first issue above.
\begin{theorem}[Pointwise properties on $|\nabla \bist_x|$; Theorems \ref{upp reg} and \ref{thm:sharp.gradient.estimate}]\label{thm:1.1}
We have the following.
\begin{enumerate}
\item{(Canonical representative of $|\nabla \bist_x|$)} For any $z \in X$, the limit;
\begin{equation}
\lim_{r\to 0^+}\intav_{B_r(z)}|\nabla \bist_x|\d\meas \in [0, \infty)
\end{equation}
exists. Denoting by $|\nabla \bist_x|^*(z)$ (or $|\nabla \bist_x|(z)$ for short if there is no confusion) the limit, we see that any point is a Lebesgue point of $|\nabla \bist_x|$, namely 
\begin{equation}
\lim_{r \to 0^+}\intav_{B_r(z)}\left| |\nabla \bist_x|-|\nabla \bist_x|^*(z)\right| \d \meas =0, \quad \forall z \in X.
\end{equation}
\item{(Upper semicontinuity)} The function $|\nabla \bist_x|^*$ is upper semicontinuous on $X$.
\item{(Sharp pointwise gradient estimate)} We have
\begin{equation}
|\nabla \bist_x|^*(z) \leqslant \mathscr C_N\nu_x^\frac{1}{N-2}, \quad \forall z \in X
\end{equation}
and
\begin{equation}
|\nabla \bist_x|^*(x)= \mathscr C_N\nu_x^\frac{1}{N-2},
\end{equation}
where 
\begin{equation}\label{sharp constcn}
\mathscr C_N:=\left(N(N-2)\right)^{\frac{1}{N-2}}.
\end{equation}
In particular $\bist_x$ with $\bist_x(x):=0$ is $\mathscr C_N\nu_x^\frac{1}{N-2}$-Lipschitz on $X$ with the (global) Lipschitz constant $\mathscr C_N\nu_x^\frac{1}{N-2}$.
\end{enumerate}
\end{theorem}
Let us introduce the second main result overcoming the second issue above, see Definition \ref{def:cone} for $N$-metric measure cones.
\begin{theorem}[Rigidity; Theorem \ref{thm:rigidity}]\label{rigigi}
If 
\begin{equation}
|\nabla \bist_x|^*(z)=\mathscr C_N\nu_x^{\frac{1}{N-2}}
\end{equation}
for some $z \in X \setminus \{x\}$, then $(X, \dist, \meas)$ is isomorphic to the $N$-metric measure cone over an $\RCD(N-2, N-1)$ space, in particular, $|\nabla \bist_x|^*\equiv \mathscr C_N\nu_x^\frac{1}{N-2}$.
\end{theorem}
Combining the results above with the compactness of non-parabolic $\RCD(0, N)$ spaces with respect to the pmGH convergence (Theorem \ref{cpt non-para}), we obtain the following almost rigidity result
\begin{theorem}[Almost rigidity; Theorem \ref{thm:almost.rigidity}]\label{thm:1.4}
For all $N >2$, $0<\eps<1$, $v>0$, $0<r<R<\infty$ and $\phi \in L^1([0,\infty), \haus^1)$ there exists $\delta:=\delta(N, \eps, r, R, \phi)>0$ such that
		if a pointed non-parabolic $\RCD(0, N)$ space $(X, \dist, \meas, x)$ satisfies (\ref{eq:67}),
	\begin{equation}
	 \frac{s}{\meas(B_s(x))} \leqslant \phi(s),\quad \text{for $\haus^1$-a.e. $s \in [1, \infty)$}
	\end{equation}
	and 
	\begin{equation}
	\mathscr C_N\nu_x^{\frac{1}{N-2}}-|\nabla \bist_x|^*(z)\leqslant \delta
	\end{equation}
	hold for some $z \in  B_R(x)\setminus B_r(x)$, then
	$(X, \dist, \meas, x)$ $\eps$-pmGH close to the $N$-metric measure cone over an $\RCD(N-2, N-1)$ space.
\end{theorem}
In particular, in Theorem \ref{thm:1.4}, if we further assume that $N$ is an integer and that the point $x$ admits an $N$-dimensional Euclidean tangent cone (which is trivial in the manifold case), then the $N$-metric measure cone stated in Theorem \ref{thm:1.4} can be replaced by the $N$-dimensional Euclidean space, which gives a positive answer to the question (Q) even in the $\RCD$ setting.

In the next subsection let us provide the outlines of the proofs of the results above.
\subsection{Outline of the proofs and organization of the paper}
In order to prove Theorem \ref{thm:1.1},  we will study a drifted Laplace operator $\mathscr L$ defined by
\begin{equation}
\mathscr Lu:= \Delta u+2\<\nabla \log G_x,\nabla u\>.
\end{equation}
Then we follow arguments by Colding in \cite{C12} to get the $\mathscr L$-subharmonicity of $|\nabla \bist_x|^2$ and the ($\Delta$-)subharmonicity of $|\nabla \bist_x|^2G(x, \cdot)$ (Proposition \ref{prop:subha}) via the Bochner inequality appearing in the definition of $\RCD$ spaces (see (\ref{bochner ineq})). Combining their subharmonicities with regularity results on subharmonic functions on PI spaces \cite{BjornBjorn} proves (1) and (2) of the theorem.
To prove the remaining statements, (3), we recall that the Green function from the pole on the $N$-metric measure cone over an $\RCD(N-2, N-1)$ space can be explicitly calculated (Proposition \ref{explicit green}) as in (\ref{s7shsnsbbsys6}) and that any tangent cone at $x$ whose $N$-volume density is finite is isomorphic to such a metric measure cone (Corollary \ref{volconemetr}) because of a result of De Philippis-Gigli \cite{DG}. 
Then, combining them with blow-up arguments at the base point $x$ based on the stability of the Laplacian \cite{AmbrosioHonda2} by Ambrosio and the first named author, we obtain (3), where the $\mathscr L$-subharmonicity of $|\nabla \bist_x|^2$ plays a role here again. 

In order to prove the rigidity result, Theorem \ref{rigigi}, we use the strong maximum principle for $\mathscr L$-subharmonic, upper semicontinuous functions to get the constancy of $|\nabla \bist_x|^2$. Then the explicit calculation of $\Delta \bist_x$ allows us to apply a rigidity result of Gigli-Violo \cite{GV} to prove Theorem \ref{rigigi}. 

Let us emphasize that under realizing the results above, we are also able to obtain a convergence result of the Green functions with respect to the pmGH convergence, in particular, as a corollary, the $W^{1,p}$-strong convergence of $\bist_x$ is proved for any finite $p<\infty$  (Corollary \ref{impr}).
After establishing compactness results on non-parabolic $\RCD(0, N)$ spaces with respect to pmGH convergence (Theorem \ref{cpt non-para}), the $W^{1,p}$-convergence result allows us to show the almost rigidity, Theorem \ref{thm:1.4}, via a contradiction. Then the main results stated in the smooth framework, Theorem \ref{thm:smooth almost} and Corollary \ref{cor main}, are corollaries of the results for RCD spaces.

Finally we show the sharpness of Theorem \ref{thm:smooth almost} via observing the $3$-metric measure cone, $C(\mathbb{S}^2(r))$ for some $r<1$ which is close to $1$, where $\mathbb{S}^2(r)$ denotes the round sphere of radius $r$ in $\mathbb{R}^3$ centered at the origin. See subsection \ref{sharp obs}.

The organization of the paper is as follows. Section \ref{sec:2} is devoted to fixing the notations/conventions and the introduction on $\RCD$ spaces, in particular, about $N$-metric measure cones.
In Section \ref{grees section}, we study the Green function on a non-parabolic $\RCD(0, N)$ space, where the starting point is a work by Bru\`e-Semola \cite{BrueSemola}. One of the main purposes in this section is to prove Theorem \ref{thm:1.1}.
Section \ref{sec:5} is devoted to proving the rigidity/almost rigidity results. 
In Section \ref{example sec}, we provide simple examples which show that our results are sharp.
In the final section, Section \ref{sec:3}, we provide proofs of regularity results about $\mathscr L$-subharmonic functions directly coming from the general theory of PI spaces \cite{BjornBjorn}. This part makes the paper to be more self-contained.

\subsection*{Acknowledgments}
The both authors would like to thank Zhangkai Huang for fruitful discussions and valuable suggestions. They also wish to thank Daniele Semola for valuable comments on the preliminary version.
Moreover we are grateful to the reviewer for his/her very careful reading of the paper and for giving us valuable suggestions for the revision, especially inspiring us to realize Corollary \ref{newcorollary}.
The first named author acknowledges supports of the Grant-in-Aid for
Scientific Research (B) of 20H01799, the Grant-in-Aid for Scientific Research (B) of 21H00977 and Grant-in-Aid for Transformative Research Areas (A) of 22H05105.
The second named author acknowledges the supports from JST SPRING Grant Number: JPMJSP2114 and JSPS KAKENHI Grant Number: JP23KJ0204.
\section{Preliminary}\label{sec:2}
\subsection{Notation and convention}
Let us fix general conventions and  geometric/analytic notations:
\begin{itemize}
\item We denote by $C(a_1,a_2,\ldots,a_k)$ a positive constant only dependent on $a_1,a_2,\ldots, a_k$, which may vary from line to line unless otherwise stated.
%
\item For a metric space $(X,\dist)$, denote by
	\begin{itemize}
		\item $B_r(x):=\{y\in X\mathop|\dist(x,y)<r\}$ and $\overline{B}_r(x):=\{y\in X\mathop|\dist(x,y)\leqslant r\}$ ; 
		\item $\Lip(X,\dist)$ the collection of all Lipschitz functions on $(X,\dist)$.
	\end{itemize}
		\item We say that a triple $(X,\dist,\meas)$ is a metric measure space if $(X,\dist)$ is a complete and separable metric space and $\meas$ is a locally finite Borel measure which is fully supported on $X$. 
	\item Whenever we discuss on a metric measure space $(X,\dist, \meas)$, we identify two objects which coincide except for a $\meas$-negligible set.
\item For a metric measure space $(X,\dist,\meas)$, let $\mathcal A$ be a collection of functions defined on an open subset $U$ of $X$. Then we denote by 
\begin{itemize}
	\item
	$\mathcal A_+:=\{f\in \mathcal A\mathop|f\geqslant  0\ \  \text{for $\meas$-a.e.}\}$;
	\item $\mathcal A_\mathrm{loc}:=\{f:U \to \mathbb{R} \mathop{|} {f\chi_{B_r(x)}\in \mathcal A} \,\,\text{for any $B_r(x)$ with $\overline{B}_r(x) \subset U$}\}$;
	\item $\mathcal A_c:=\{f\in \mathcal A\mathop|\text{The support of $f$, $\supp f$, is compact and is included in $U$}\}.$
\end{itemize}
For instance, $L^p_\mathrm{loc}(U, \dist, \meas)$, $\Lip_c(U,\dist)$, $W^{1,p}_{\mathrm{loc}}(U, \dist, \meas)$, etc.,  make sense. 
\end{itemize}
\subsection{Definition of $\RCD(K,N)$ space and heat kernel}
Let $(X,\dist,\meas)$ be a metric measure space. 
We define the local Lipschitz constant at $x$ of a function $f$ defined on $X$ as follows:
\begin{equation}
\lip f(x):=\limsup_{y\ra x}\frac{|f(x)-f(y)|}{\dist(x,y)},
\end{equation}
where $\lip f(x)$ is intepreted as $0$ if $x$ is isolated.
For any $f\in L^{2}(X, \meas)$, the Cheeger energy of $f$ is defined by
\begin{equation}\label{eq:02}
\Ch(f):=\inf\left\{\frac 1 2\liminf_{i\ra\infty}\int_{X}(\lip f)^2\ \d\meas\ \bigg|\ f_i\in\Lip (X,\dist)\cap (L^{\infty} \cap L^2)(X, \meas), \|f_i-f\|_{L^2}\ra 0\right\}.
\end{equation}
The Sobolev space $W^{1,2}(X,\dist,\meas)$\footnote{Similarly we can define $W^{1,p}(X, \dist, \meas)$. See for instance \cite{AmbrosioHonda}.} is the collection of $L^2(X, \meas)$-functions with finite Cheeger energy, equipped with the $W^{1,2}$-norm
\begin{equation}
\|f\|_{W^{1,2}(X,\dist,\meas)}:=\sqrt{\|f\|_{L^2}^2+2\Ch(f)}.
\end{equation}
For any $f \in W^{1,2}(X, \dist, \meas)$, by taking a minimizing sequence $\{f_i\}_i$ in the right-hand-side of \eqref{eq:02}, we can find the optimal $L^2$-function denoted by $|\nabla f|$, called the minimal relaxed slope of $f$, realizing the Cheeger energy, namely
\begin{equation}
\Ch(f)=\frac 1 2\int_X|\nabla f|^2\ \d\meas.
\end{equation}

	We say that $(X,\dist,\meas)$ is infinitesimally Hilbertian if $W^{1,2}(X,\dist,\meas)$ is a Hilbert space. In this case, we set
	\begin{equation}
	\<\nabla f_1,\nabla f_2\>:=\lim_{t\ra 0}\frac{|\nabla (f_1+tf_2)|^2-|\nabla f_1|^2}{2t}\in L^1(X, \meas), \ \forall f_1,f_2\in W^{1,2}(X,\dist,\meas),
	\end{equation}
	which is symmetric and bi-linear in $\meas$-a.e. sense (see for instance \cite[Theorem 4.3.3]{GP2} for several equivalent definitions of infinitesimal Hilbertianity). 
Moreover then we can define the (linear) Laplacian as follows;
	we denote by $D(\Delta)$ the set of all $f \in W^{1,2}(X, \dist, \meas)$ such that there exists $h\in L^2(X, \meas)$ such that
	\begin{equation}
	-\int_{X}\<\nabla f,\nabla\phi\>\ \d\meas=\int_{X}\phi h\ \d\meas, \quad \forall \phi\in W^{1,2}(X,\dist,\meas).
	\end{equation}
	Since such $h$ is unique whenever it exists, we shall denote by $\Delta f$. 
	We are now in a position to give the definition of $\RCD$ spaces.
\begin{definition}[$\RCD$-space]
	We say that $(X,\dist,\meas)$ is an $\RCD(K,N)$ \textit{space} for some $K\in\R$ and $N\geqslant 1$ if the following four conditions are satisfied.
	\begin{enumerate}
		\item(Volume growth bound) There exist $C>0$ and $x \in X$ such that
		\begin{equation}
		\meas(B_r(x))\leqslant Ce^{Cr^2}, \quad \forall r>0.
		\end{equation} 
		\item(Infinitesimal Hilbertianity) $(X, \dist, \meas)$ is infinitesimally Hilbertian.
		\item(Sobolev-to-Lipschitz property) Any $f\in W^{1,2}(X,\dist,\meas)$ with $|\nabla f|\leqslant L$ for $\meas$-a.e. admits an $L$-Lipschitz representative.
		\item(Bochner's inequality) For any $f\in D(\Delta)$ with $\Delta f\in W^{1,2}(X,\dist,\meas)$ and any $\phi\in D(\Delta)\cap L^\infty_+(X, \meas)$ with $\Delta \phi \in L^\infty(X, \meas)$, it holds that
		\begin{equation}\label{bochner ineq}
		\frac 1 2\int_X \Delta\phi|\nabla f|^2\ \d\meas\geqslant\int_X \phi\left(\frac{(\Delta f)^2}{N}+\<\nabla f,\nabla\Delta f\>+K|\nabla f|^2\right)\ \d\meas.
		\end{equation}
	\end{enumerate}
	\end{definition}
There are also several equivalent characterizations of $\RCD(K,N)$-conditions, see \cite{AmbrosioMondinoSavare, CavallettiMilman, ErbarKuwadaSturm}. We refer to \cite{A} as a good survey for the theory of $\RCD$ spaces.

Let us also mention that there exist local notions above, including the domain $D(\Delta, U)$ of the local Laplacian defined on an open subset $U$ of $X$, the local Sobolev space $W^{1,2}(U, \dist, \meas)$ and so on.\footnote{For instance, for any $f \in D(\Delta, U)$ we have $\Delta f, |\nabla f|, f \in L^2(U, \meas)$.} In the sequel we immediately use them, see for instance \cite{AmbrosioHonda2, BjornBjorn, HKST} for the details.

We here recall the precise definitions of the heat flow and the heat kernel on an $\RCD(K, N)$ space $(X, \dist, \meas)$ for some $K \in \mathbb{R}$ and some finite $N \geqslant 1$.
For any $f \in L^2(X, \meas)$, there exists a unique locally absolutely continuous (or equivalently, smooth, in this setting, (see \cite{GP2})) curve $h_{\cdot}f:(0, \infty) \to L^2(X, \meas)$, called the heat flow starting at $f$, such that $h_tf \to f$ in $L^2(X, \meas)$ as $t \to 0^+$ and that $h_tf \in D(\Delta)$ for any $t>0$ with
\begin{equation}\label{heat equa}
\frac{\d}{\d t}h_tf=\Delta h_tf.
\end{equation}

Then, thanks to \cite{Sturm06,Sturm06b} with the Bishop-Gromov inequality and the Poincar\'e inequality which will be explained in the next subsection \ref{subsec:2.2},
the heat flow can be written by the integral of a unique continuous kernel  $p=p_X:X \times X \times (0, \infty) \to (0, \infty)$, called the heat kernel of $(X,\dist,\meas)$. Namely,
	 for all $f\in L^2(X, \meas)$, we have $h_tf \in C(X)$ with
	\begin{equation}
	\h_tf(x)=\int_X p(x,y,t)f(y)\ \d\meas(y), \quad \forall x \in X.
	\end{equation}
	Note that the heat kernel $p$ can be characterized by using the dual heat flow $\tilde h_t$ acting on the space of all Borel probability measures with finite quadratic moments $\mathcal{P}_2(X)$;
	\begin{equation}\label{heatdual}
	\tilde h_t\mathbf{\delta}_x=p(x, \cdot, t)\meas,\quad \forall x \in X, \quad \forall t>0,
	\end{equation}
	where $\mathbf{\delta}_x$ is the Dirac measure at $x$.
	

Let us write a formula on the heat kernel under a rescaling, which directly follows from the definition; for all $a, b >0$, the $\RCD(a^{-2}K, N)$ space
\begin{equation}(\tilde X, \tilde \dist, \tilde \meas) :=\left(X, a\dist, b\meas\right)
\end{equation}
satisfies
\begin{equation}\label{rescaling heat}
p_{\tilde X}(x, y, t)=\frac{1}{b}p_X(x, y, a^{-2}t).
\end{equation}

In order to keep our presentation short, we assume that the readers are familiar with basics on the $\RCD$ theory, including pointed measured Gromov-Hausdorff (pmGH) convergence, its metrization $\dist_{\mathrm{pmGH}}$, stability/compactness of $\RCD$ spaces with respect to $\dist_{\mathrm{pmGH}}$, and functional convergence with respect to $\dist_{\mathrm{pmGH}}$. We refer a recent nice survey \cite{GIGLI}
about this topic (see also \cite{AmbrosioHonda, AmbrosioHonda2, GigliMondinoSavare13}).

Let us end this subsection by introducing the following two notions with related results. 
\begin{definition}[Tangent cone]
	A pointed $\RCD(0,N)$ space $(Y,\dist_Y,\meas_Y,y)$ is said to be a \textit{tangent cone of $(X,\dist,\meas)$ at $x \in X$} (or \textit{tangent cone at infinity of $(X, \dist, \meas)$} in the case when $K=0$, respectively) if there exists a sequence $r_i\to 0^+$ (or $r_i \to \infty$, respectively) such that
	\begin{equation}
		\left(X,\frac 1{r_i}\dist,\frac\meas{\meas(B_{r_i}(x_i))},x\right)\xrightarrow{\mathrm{pmGH}}(Y,\dist_Y,\meas_Y,y).
	\end{equation}
	Moreover a point $x \in X$ is called \textit{$k$-regular} if any tangent cone at $x$ is isomorphic to the $k$-dimensional Euclidean space $(\mathbb{R}^k, \dist_{\mathrm{Euc}}, \omega_k^{-1}\haus^k, 0_k)$.
	\end{definition}
	\begin{remark}
	We often use $\frac{\meas}{r_i^k}$ for some $k \geqslant 1$ instead of using $\frac{\meas}{\meas (B_{r_i}(x))}$ in the definition above, and we also call such limit a tangent cone. 
	\end{remark}
	It is proved in \cite{BrueSemola}
if $X$ is not a single point, then there exists a unique integer $k$ at most $N$ such that for $\meas$-a.e. $x \in X$, $x$ is $k$-regular. We call $k$ the essential dimension of $(X, \dist, \meas)$ (see also \cite{CN, Deng}). 
It is known that the essential dimension is at most the Hausdorff dimension, however in general they do not coincide. See \cite{PW}.
	The following is defined in \cite{DePhillippisGigli} as a synthetic couterpart of volume non-collapsed Ricci limit spaces.
	\begin{definition}[Non-collapsed space]\label{ncrcddef}
	We say that $(X, \dist, \meas)$ is \textit{non-collapsed} if $\meas=\haus^N$.
	\end{definition}
	It is known that any non-collapsed $\RCD$ space has nicer properties rather than that of general $\RCD$ spaces, including a fact that $N$ must be an integer, and the Bishop inequality in the case when $K=0$;
	\begin{equation}\label{bishop ine}
	\frac{\haus^N(B_r(x))}{\omega_Nr^N}\leqslant 1,\quad \forall r>0.
	\end{equation}
	It is worth mentioning that $(X, \dist, \meas)$ is non-collapsed, up to multiplying a positive constant to the reference measure, if the essential dimension is equal to $N$, or $N$ is an integer with the existence of an $N$-regular point. See \cite{BGHZ} (and \cite{H}).
\subsection{Geometric and analytic inequalities on $\RCD(0, N)$ spaces}\label{subsec:2.2}
Let us recall several inequalities on an $\RCD(0, N)$ space $(X, \dist, \meas)$ for some finite $N \ge 1$. Fix $x \in X$.
The Bishop-Gromov inequality states 
\begin{equation}\label{bg ineq}
		\frac{\meas(B_{r}(x))}{r^N}\geqslant\frac{\meas(B_{s}(x))}{s^N}, \quad \forall r <s.
		\end{equation}
		See \cite{LottVillani} and \cite{Sturm06,Sturm06b} for the proof. Based on this inequality we introduce;

\begin{definition}[$N$-volume density and asymptotic $N$-volume]\label{Nvlo}
The \textit{$N$-volume density at $x$}, denoted by $\nu_x$, is defined by
\begin{equation}\label{volume density}
\nu_x:=\lim_{r\ra 0}\frac{\meas(B_r(x))}{r^N}\in(0,\infty].
\end{equation} 
Moreover the \textit{asymptotic $N$-volume}, denoted by $V_X$, is defined by
\begin{equation}\label{volume density1}
V_X:=\lim_{r\ra \infty}\frac{\meas(B_r(x))}{r^N}\in[0,\infty).
\end{equation}
Note that the Bishop-Gromov inequality (\ref{bg ineq}) implies the existence of the both right-hand-sides of (\ref{volume density}) and (\ref{volume density1}), that $V_X$ does not depend on the choice of $x\in X$ and that $\nu_x\ge V_X$.
\end{definition}
\begin{remark}
Let us provide a formula  on the $N$-volume density under a rescaling; for all $a, b >0$, the $N$-volume density $\tilde \nu_{\tilde x}$ of the pointed non-parabolic $\RCD(0, N)$ space $(\tilde X, \tilde \dist, \tilde \meas, \tilde x) :=(X, a\dist, b\meas, x)$
satisfies
\begin{equation}
\tilde \nu_{\tilde x}=\lim_{r \to 0^+}\frac{b\meas (\tilde B_r(x))}{r^N}=\frac{b}{a^N}\lim_{r \to 0^+}\frac{\meas (B_{a^{-1}r}(x))}{(a^{-1}r)^N}=\frac{b}{a^N}\nu_x.
\end{equation}
In particular $\tilde \nu_{\tilde x}=\nu_x$ if $b=a^N$, which will play a role later. Note that similarly we have $V_{\tilde X}=\frac{b}{a^N}V_X$.
\end{remark}
Next let us recall Gaussian estimates on the heat kernel $p$ established in \cite{JiangLiZhang}; for any $0<\eps<1$, there exists $C(N, \eps)>1$ such that
\begin{equation}\label{eq:gaussian}
\frac{1}{C(N, \eps) \meas (B_{t^{1/2}}(x))}\exp \left(-\frac{\dist^2 (x, y)}{4(1-\epsilon)t}\right) \leqslant p(x, y, t) \leqslant \frac{C(N, \eps)}{\meas (B_{t^{1/2}}(x))}\exp \left( -\frac{\dist^2 (x, y)}{4(1+\epsilon)t}\right)
\end{equation}
and
\begin{equation}\label{eq:equi lip}
|\nabla_x p(x, y, t)|\leqslant \frac{C(N, \eps)}{t^{1/2}\meas (B_{t^{1/2}}(x))}\exp \left(-\frac{\dist^2(x, y)}{(4+\epsilon)t}\right).
\end{equation}

Finally let us recall the following Poincar\'e inequality proved in \cite{Rajala}; 
\begin{equation}\label{poincareineq}
\intav_{B_r(x)}\left|f-\intav_{B_r(x)}f\ \d\meas\right|\ \d\meas \leqslant 4r\intav_{B_{2r}(x)}|\nabla f|\ \d\meas,\quad \forall r>0,\ \forall f\in W^{1,2}(X,\dist,\meas),
\end{equation}
where 
\begin{equation}
\intav_A\cdot \ \d\meas:=\frac 1{\meas(A)}\int_A\cdot\ \d\meas
\end{equation}
denotes the integral average for any measurable set $A$ with positive and finite measure. 

When $N>2$, the Poincar\'e inequality combining with the Bishop-Gromov inequality (\ref{bg ineq}) implies the self-improved Poincar\'e inequality:
\begin{equation}
\left(\intav_{B_r(x)}\left|f-\intav_{B_r(x)}f\ \d\meas\right|^\frac{2N}{N-2}\ \d\meas\right)^\frac{N-2}{2N} \leqslant C(N)r\left(\intav_{B_{2r}(x)}|\nabla f|^2\ \d\meas\right)^\frac 12.
\end{equation}
Moreover if  $f\in W^{1,2}_0(B_r(x),\dist,\meas)$ and $X$ is non-compact, then we have a more convenient corollary usually referred as Sobolev inequality:
\begin{equation}\label{Sobolevineq}
\left(\intav_{B_r(x)}|f|^\frac{2N}{N-2}\ \d\meas\right)^\frac{N-2}{2N}\leqslant C(N)r\left(\intav_{B_r(x)}|\nabla f|^2\ \d\meas\right)^\frac 1 2
\end{equation}
which plays a central role to get various properties on differential operators including the Laplacian and
 a drifted Laplace operator $\mathscr{L}$ in Section \ref{sec:3}. See for instance \cite{BjornBjorn, HK} for the details.
\subsection{Metric measure cone and  rigidity}
In this subsection, we introduce known rigidity results to an $N$-metric measure cone whose definition is as follows. In the sequel, we fix a finite $N >2$.
\begin{definition}[$N$-metric measure cone]\label{def:cone}
	The \textit{($N$-)metric measure cone} $(C(Y),\dist_{C(Y)},\meas_{C(Y)})$ over
	 an $\RCD(N-2,N-1)$ space $(Y,\dist_Y,\meas_Y)$ 
	 is defined by 
	\begin{align}
	C(Y):=& [0,\infty)\times Y/(\{0\}\times Y),\nonumber\\
	\dist_{C(Y)}\left((r_1,y_1),(r_2,y_2)\right):=&\sqrt{r_1^2+r_2^2-2r_1r_2\cos(\dist_Y(y_1,y_2))},\nonumber\\
	\d\meas_{C(Y)}(r,y):=&r^{N-1}\d r\otimes\d\meas_Y(y),
	\end{align}
	where $\d r=\mathcal L^1$ is the $1$-dimensional Lebesgue measure. Denote by $O_Y:=[(0, y)]$ the pole of $C(Y)$. 
	\end{definition}
	Note that in the definition above, if $(Y, \dist_Y, \meas_Y)$ is non-collapsed, then $(C(Y), \dist_{C(Y)}, \meas_{C(Y)})$ is also non-collapsed because we can easily check by definition
	\begin{equation}
	\lim_{s \to 0^+}\frac{\meas_{C(Y)}(B_s(r,x))}{\omega_{N+1}s^{N+1}}=1,\quad \text{for any $r>0$ and any $(N-1)$-regular point $x$ of $X$.}
	\end{equation}
	This remark will play a role in subsection \ref{sharp obs}.
	
The following results are fundamental results for $N$-metric measure cones, where (1) is due to \cite[Corollary 1.3]{Ketterer2} and (2) is obtained in \cite[Theorem 1.1]{DG} and  \cite[Theorem 5.1]{GV}.
\begin{theorem}[Rigidity]\label{thm:rigidity cone}
We have the following.
\begin{enumerate}
\item The $N$-metric measure cone over an $\RCD(N-2, N-1)$ space is an $\RCD(0, N)$ space.
\item Let $(X, \dist, \meas, x)$ be a pointed $\RCD(0, N)$ space. Then the following three conditions are equivalent:
\begin{enumerate}
\item $(X, \dist, \meas, x)$ is isomorphic to the $N$-metric measure cone over an $\RCD(N-2, N-1)$ space with the pole; 
\item there exists $f \in D_{\mathrm{loc}}(\Delta)$ such that $\Delta f=2N$ holds (in particular $f$ must be locally Lipschitz because of \cite{AmbrosioMondinoSavare16, Jiang}), that $f$ is positive on $X \setminus \{x\}$ with $f(x)=0$, and that $|\nabla \sqrt{2f}|^2=1$ (moreover then $f$ is equal to $\frac{1}{2}\dist(x, \cdot)^2$);
\item the function
\begin{equation}
R \mapsto \frac{\meas(B_R(x))}{R^N}
\end{equation}
is constant.
\end{enumerate}
\end{enumerate}
\end{theorem}
The following is a well-known result which will play a central role in the paper.
\begin{corollary}\label{volconemetr}
We have the following.
\begin{enumerate}
\item If an $\RCD(0, N)$ space $(X, \dist, \meas)$ has the finite $N$-volume density $\nu_x<\infty$ at a point $x \in X$, then any tangent cone at $x$ is isomorphic to the $N$-metric measure cone over an $\RCD(N-2, N-1)$ space.
\item If an $\RCD(0, N)$ space $(X, \dist, \meas)$ has the positive asymptotic $N$-volume $V_X>0$, then any tangent cone at infinity is isomorphic to the $N$-metric measure cone over an $\RCD(N-2, N-1)$ space.
\end{enumerate}
\end{corollary}
\begin{proof}
We give only a proof of (1) because (2) is similar.
Denoting by $\tilde \meas, \tilde x$ the reference measure, the base point, respectively on a tangent cone, it easily follows from the finiteness $\nu_x<\infty$ that
\begin{equation}
\frac{\tilde \meas (B_R(\tilde x))}{\tilde \meas (B_r(\tilde x))}=\left(\frac{R}{r}\right)^N, \quad \forall r>0, \quad \forall R>0.
\end{equation}
Thus (2) of Theorem \ref{thm:rigidity cone} allows us to conclude.
\end{proof}
Let us recall the explicit description on the heat kernel on an $N$-metirc measure cone.
\begin{proposition}\label{expl heat}
Let $(X, \dist, \meas, x)$ be isomorphic to the $N$-metric measure cone  with the pole over an $\RCD(N-2, N-1)$ space. Then we have
\begin{equation}\label{nu equiv}
\frac{\meas (B_r(x))}{r^N}=\meas (B_1(x))=\nu_x=V_X, \quad \forall r>0
\end{equation}
and
\begin{equation}\label{ex heat}
p(x,y, t)=Ct^{-\frac{N}{2}}\exp \left(-\frac{\dist(x, y)^2}{4t}\right),\quad\forall y \in X,
\end{equation}
where 
\begin{equation}
C=\frac{2^{1-N}}{N\Gamma \left(\frac{N}{2}\right)\meas (B_1(x))} 
\end{equation}
\end{proposition}
\begin{proof}
It follows by Definitions \ref{Nvlo} and \ref{def:cone} that (\ref{nu equiv}) holds.
On the other hand, (\ref{ex heat}) is a direct consequence of \cite[Proposition 4.10]{KD} with (\ref{heatdual}) and (2) of Theorem \ref{thm:rigidity cone}. See \cite[Proposition 2.13]{Huang} for a more general result (see also \cite[Theorem 6.20]{Ding} and \cite[Section 8]{Taylor}).
\end{proof}
\section{Green function}\label{grees section}
In this section we discuss the Green function on a non-parabolic $\RCD(0, N)$ space.
\subsection{Non-parabolic $\RCD(0, N)$ space}
Throughout the section, we fix an $\RCD(0, N)$ space $(X, \dist, \meas)$ for some finite $N >2$, which is not necessarily an integer.
Let us start by introducing the following fundamental notion due to \cite{BrueSemola} in our framework.
\begin{definition}[Non-parabolic $\RCD(0,N)$ space]
	$(X, \dist, \meas)$ is said to be \textit{non-parabolic} if for some point $x\in X$ (and thus for any $x\in X$), 
		\begin{equation}
	\label{eq:non-parabolic}
	\int_{1}^\infty \frac{s}{\meas\big(B_s(x)\big)}\ \d s<\infty.
	\end{equation}
	\end{definition}
It is trivial that $(X, \dist, \meas)$ is non-parabolic if $V_X>0$ because of the Bishop-Gromov inequality.
In the sequel we assume that $(X, \dist, \meas)$ is non-parabolic. Note that the diameter must be infinite, thus it is non-compact.  Then we can define the Green function as follows.
		
\begin{definition}[Green function]
	\label{def:green.function}
	The  \textit{Green function} $G=G^X$ of the non-parabolic $\RCD(0, N)$ space $(X, \dist, \meas)$ is defined by
	\begin{align}
	\label{eq:green.function}
	G:&\ X\times X\setminus\mathrm{diag}(X)\ra (0, \infty),\nonumber\\ 
	&\ (x,y)\mapsto\int_{0}^{\infty}p(x,y,t)\ \d t,
	\end{align}
	where $\mathrm{diag}(X):=\{(x,x) \in X \times X | x \in X\}$. In the following we write $G_x^X(\cdot)=G_x(\cdot):=G(x,\cdot):X\setminus\{x\}\ra (0, \infty)$.
	\end{definition}
	In the sequel, we fix $x \in X$. It is proved in (the proof of) \cite[Lemma 2.5]{BrueSemola} that $G_x$ is harmonic on $X \setminus \{x\}$ and that $G_x \in  W^{1,1}_{\mathrm{loc}}(X,\dist,\meas)$ holds with
\begin{equation}
\int_X \Delta f(y)G_x(y)\ \d\meas(y)=-f(x)
\end{equation}
for any $f \in D(\Delta)$ with $\Delta f \in L^{\infty}(X, \meas)$.
In order to introduce quantitative estimates on $G_x$, let us prepare the following auxiliary functions for all $x\in X$ and $r\in(0,\infty)$;
\begin{equation}\label{deff}F_x^X(r)=F_x(r):=\int_{r}^\infty \frac{s}{\meas\big(B_s(x)\big)}\ \d s, \quad H_x^X(r)=H_x(r):=\int_{r}^\infty \frac{1}{\meas\big(B_s(x)\big)}\ \d s.	\end{equation}
It is easy to see that both $F$ and $H$ are continuous with respect to the two variables $(x, r) \in X \times (0, \infty)$. 
Note that Bishop-Gromov inequality (\ref{bg ineq}) shows for any $r>0$
\begin{equation}\label{ps9s8shbbbsn}
\frac{1}{(N-2)\nu_x} \leqslant \frac{F_x(r)}{r^{2-N}}\leqslant \frac{1}{(N-2)V_X}, \quad \frac{1}{(N-1)\nu_x}\leqslant \frac{H_x(r)}{r^{1-N}}\leqslant \frac{1}{(N-1)V_X}.
\end{equation}
Let us provide formulae on their asymptotics.
\begin{lemma}
	\label{lem:asymptotic.F.and.H}
	The following asymptotic properties hold as $r\ra 0$:
	\begin{equation}
	\label{eq:asymptotic.F}
	\lim_{r\ra 0}\frac{F_x(r)}{r^{2-N}}=\frac{1}{(N-2)\nu_x},
	\end{equation} 
	\begin{equation}
	\label{eq:asymptotic.H}
	\lim_{r\ra 0}\frac{H_x(r)}{r^{1-N}}=\frac{1}{(N-1)\nu}_x.
	\end{equation}
	Moreover the following asymptotic properties hold as $r \to \infty$:
\begin{equation}\label{limit asympt F}
\lim_{r \to \infty}\frac{F_x(r)}{r^{2-N}}=\frac{1}{(N-2)V_X},
\end{equation}
\begin{equation}\label{limit asympt H}
\lim_{r \to \infty}\frac{H_x(r)}{r^{1-N}}=\frac{1}{(N-1)V_X},
\end{equation}
where the limits of (\ref{limit asympt F}) and of (\ref{limit asympt H}) can be understood as $\infty$ in the case when $V_X=0$.
	\end{lemma}
\begin{proof}
	We only show \eqref{eq:asymptotic.F} under assuming $\nu_x<\infty$ because the others can follow from similar arguments. The Bishop-Gromov inequality (\ref{bg ineq}) shows that the map
	\begin{equation}
	I_x(r):=\nu_x-\frac{\meas(B_r(x))}{r^N}
	\end{equation}
	is non-decreasing with $I_x(r)\geqslant 0$. Thus fixing $r_0>0$, we have for any $0<r<r_0$, 
	\begin{equation}\label{rr0}
	\nu_x r^N-I_x(r_0)r^N\leqslant\meas(B_r(x))\leqslant \nu_x r^N.
	\end{equation}
	Write $F_x(r)$ as
	\begin{equation}
	F_x(r)=\int_{r}^{r_0}\frac{s}{\meas(B_s(x))}\ \d s+\int_{r_0}^\infty\frac{s}{\meas(B_s(x))}\ \d s. 
	\end{equation}
	Then by (\ref{rr0}) the first term of the right-hand-side can be estimated as follows.
	\begin{equation}
	\frac {r_0^{2-N}-r^{2-N}}{(2-N)\nu_x}\leqslant \int_r^{r_0} \frac{s}{\meas(B_s(x))}\ \d s\leqslant \frac{r_0^{2-N}-r^{2-N}}{(2-N)(\nu_x-I_x(r_0))}.
	\end{equation}
	Thus
	\begin{equation}
	\begin{aligned}
	r^{N-2}\int_{r_0}^\infty&\frac{s}{\meas(B_s(x))}\ \d s-C(N,\nu_x)\left(\frac{r}{r_0}\right)^{N-2}\leqslant\frac{F_x(r)}{r^{2-N}}-\frac{1}{(N-2)\nu_x}\\&\leqslant r^{N-2}\int_{r_0}^\infty\frac{s}{\meas(B_s(x))}\ \d s+\frac {C(N,\nu_x)}{I_x(r_0)}\left(\frac{r}{r_0}\right)^{N-2}.
	\end{aligned}
	\end{equation}
	For any $\eps\in(0,1)$, we can let $r=\eps r_0$ and thus
	\begin{equation}
	\begin{aligned}\label{quant F}
		(\eps r_0)^{N-2}\int_{r_0}^\infty&\frac{s}{\meas(B_s(x))}\ \d s-C(N,\nu_x)\eps^{N-2}\leqslant\frac{F_x(\eps r_0)}{(\eps r_0)^{2-N}}-\frac{1}{(N-2)\nu_x}\\ 
		&\leqslant (\eps r_0)^{N-2}\int_{r_0}^\infty\frac{s}{\meas(B_s(x))}\ \d s+\frac{C(N,\nu_x)}{I_x(r_0)}\eps^{N-2}.
		\end{aligned}
	\end{equation}
	Letting $\eps\ra 0$ completes the proof of (\ref{eq:asymptotic.F}).
\end{proof}
We are now in a position to introduce estimates on $G$ by $F, H$ given in \cite[Proposition 2.3]{BrueSemola} after \cite{Gri} in the smooth setting.
\begin{proposition}
	\label{thm:estimate.Green.function.in.BS}
	There exists $C=C(N)>1$ such that
	\begin{equation}
	\label{eq:BS.estimate.F}
	\frac{1}{C}F_x(\dist(x,y))\leqslant G_x(y)\leqslant CF_x(\dist(x,y)), \quad \forall y\in X\setminus \{x\},
	\end{equation}
 and 
	\begin{equation}
	\label{eq:BS.estimate.H}
	|\nabla G_x|(y)\leqslant CH_x(\dist(x,y)),\quad \text{for $\meas$-a.e. } y\in X.
	\end{equation}
\end{proposition}
Since $F_x(r)\ra 0$ as $r\ra\infty$, we have immediately the following. 
\begin{corollary}
	\label{coro:vanish.at.infinity}
	We have $G_x(y)\ra 0$ as $\dist(x,y)\ra\infty$.
\end{corollary}
Let us define the main target on the paper.
\begin{definition}[Smoothed distance function $\bist_x$]\label{modified}
Define a function $\bist_x^X=\bist_x$ on $X \setminus \{x\}$ by 
\begin{equation}
\bist_x:=G_x^{\frac 1{2-N}}.
\end{equation}
\end{definition}
We provide formulae related to $\bist_x$, which will play roles later.
\begin{lemma}\label{lem form}
We have 
	\begin{equation}
	\Delta \bist_x=(N-1)\frac{|\nabla \bist_x|^2}{\bist_x},
	\end{equation}
	and 
	\begin{equation}
	\label{eq:6}
	\Delta \bist_x^2=2N|\nabla \bist_x|^2.
	\end{equation}
\end{lemma} 
\begin{proof} Because
	\begin{equation}
	\Delta \bist_x=\Delta G_x^{\frac{1}{2-N}}=\frac{N-1}{(N-2)^2}G_x^{\frac{2N-3}{2-N}}|\nabla G_x|^2=(N-1)\frac{|\nabla \bist_x|^2}{\bist_x},
	\end{equation}
	and thus
	\begin{equation}
	\Delta \bist_x^2=2\bist_x\Delta \bist_x+2|\nabla \bist_x|^2= 2N|\nabla \bist_x|^2,
	\end{equation}
	where we used the fact that $G_x$ is harmonic on $X \setminus \{x\}$.
\end{proof}
We give the explicit formula for the smoothed distance function for an $N$-metric measure cone. Although this is well-known (see for instance \cite[Lemma 2.7]{BrueDengSemola}), let us provide a proof for readers' convenience.
\begin{proposition}[Green function on $N$-metric measure cone]\label{explicit green}
If $(X, \dist, \meas, x)$ is isomorphic to the $N$-metric measure cone with the pole over an $\RCD(N-2, N-1)$ space, then we have
\begin{equation}
G_x(y)=\frac{1}{N(N-2)\meas (B_1(x))}\dist(x, y)^{2-N},\quad \forall y \in X \setminus \{x\}.
\end{equation}
In particular
\begin{equation}\label{identity b}
\bist_x=\mathscr C_N\nu_x^{\frac{1}{N-2}}\dist_x,
\end{equation}
where
\begin{equation}\label{optimalc}
\mathscr C_N:=\left(N(N-2)\right)^{\frac{1}{N-2}}.
\end{equation}
\end{proposition}
\begin{proof}
Thanks to Proposition \ref{expl heat}, we know
\begin{align}
G_x(y)&=\frac{2^{1-N}}{N\Gamma \left(\frac{N}{2}\right)\meas (B_1(x))}\int_0^{\infty}t^{-\frac{N}{2}}\exp \left(-\frac{\dist(x, y)^2}{4t}\right)\d t \nonumber \\
&=\frac{2^{1-N}}{N\Gamma \left(\frac{N}{2}\right)\meas (B_1(x))} \cdot \dist(x,y)^{2-N}4^{\frac{N}{2}-1}\Gamma \left(\frac{N}{2}-1\right) \nonumber \\
&=\frac{1}{N(N-2)\meas (B_1(x))}\dist(x, y)^{2-N}.
\end{align}
Finally recalling (\ref{nu equiv}), we get (\ref{identity b}).
\end{proof}
Finally let us provide formulae on the functions above, $F, G, H$ and $\bist$ under rescalings, which will play roles later.
\begin{lemma}\label{rescaling fh}
For all $a, b>0$, consider the  rescaled non-parabolic $\RCD(0, N)$ space;
\begin{equation}
\left(\tilde X, \tilde \dist, \tilde \meas, \tilde x\right):=\left(X, a\dist, b\meas, x\right).
\end{equation}
Then the Green function $G^{\tilde X}_{\tilde x}$, 
the corresponding auxiliary functions $F^{\tilde X}_{\tilde x}, H^{\tilde X}_{\tilde x}$, and the smoothed distance function $\bist^{\tilde X}_x$ of the rescaled space
satisfy
\begin{equation}\label{rescaling F}
G^{\tilde X}_{\tilde x}=\frac{a^2}{b}G_x^X, \quad F^{\tilde X}_{\tilde x}(r)=\frac{a^2}{b}F_x^X\left(\frac{r}{a}\right),\quad H^{\tilde X}_{\tilde x}(r)=\frac{a}{b}H_x^X\left(\frac{r}{a}\right)
\end{equation}
and 
\begin{equation}\label{rescaling b}
\bist_{\tilde x}^{\tilde X}=\frac{a^{\frac{2}{2-N}}}{b^{\frac{1}{2-N}}}\bist_x^X,\quad |\tilde{\nabla}\bist_{\tilde x}|=\frac{a^{\frac{N}{2-N}}}{b^{\frac{1}{2-N}}} |\nabla \bist_x|.
\end{equation}
In particular, if  $b=a^N$, then
\begin{equation}\label{rescaling F}
G^{\tilde X}_{\tilde x}(y)=a^{2-N}G^X_x(y),\quad F^{\tilde X}_{\tilde x}(r)=a^{2-N}F_x\left(\frac{r}{a}\right),\quad H^{\tilde X}_{\tilde x}(r)=a^{1-N}H_x\left(\frac{r}{a}\right).
\end{equation}
and 
\begin{equation}
\bist_{\tilde x}=a\bist_x,\quad |\tilde{\nabla}\bist_{\tilde x}|=|\nabla \bist_x|.
\end{equation}
\end{lemma}
\begin{proof}
The formula for $G$ is a direct consequence of (\ref{rescaling heat}). Moreover it implies (\ref{rescaling b}).
On the other hand,
since 
\begin{align}
		F_{\tilde x}(r)=\int_r^{\infty}\frac{s}{b\meas(\tilde B_s(\tilde x))}\d s&=\frac{1}{b}\int_r^{\infty}\frac{s}{\meas (B_{a^{-1}s}(x))}\d s\nonumber \\
		&=\frac{a^2}{b}\int_{a^{-1}r}^{\infty}\frac{t}{\meas (B_t(x))}\d t=\frac{a^2}{b}F_x\left(\frac{r}{a}\right),
		\end{align}
		we have the desired formula for $F$. Similarly we have the remaining results.
\end{proof}
\subsection{Convergence}
In this subsection we discuss the convergence of non-parabolic $\RCD(0, N)$ spaces with respect to the pmGH topology. 
Let us introduce an elementary lemma.
\begin{lemma}
Let 
\begin{equation}
(X_i, \dist_i, \meas_i, x_i)\stackrel{\mathrm{pmGH}}{\to}(X, \dist, \meas, x)
\end{equation}
be a pmGH convergent sequence of $\RCD(0, N)$ spaces for some finite $N \geqslant 1$. Then we have 
\begin{equation}\label{lower semicont}
\liminf_{i \to \infty}\nu_{x_i}\geqslant \nu_{x_{\infty}}
\end{equation}
and 
\begin{equation}\label{upper asym}
\limsup_{i \to \infty}V_{X_i}\le V_X.
\end{equation}
\end{lemma}
\begin{proof}
We give only a proof of (\ref{upper asym}) because the proof (\ref{lower semicont}) is similar (moreover this is valid even in the case of negative lower bounds on Ricci curvature.  See also \cite[Subsection 2.3]{LW}).
For fixed $r>0$, we have
\begin{equation}
V_{X_i} \le \frac{\meas_i(B_r(x_i))}{r^N} \to \frac{\meas (B_r(x))}{r^N}
\end{equation}
which shows
\begin{equation}
\limsup_{i \to \infty}V_{X_i} \le \frac{\meas(B_r(x))}{r^N}.
\end{equation}
Then letting $r \to \infty$ completes the proof of (\ref{upper asym}).
\end{proof}

Next let us provide a compactness result as follows. In the sequel we fix a finite $N >2$. Note that if $F_x(1)\le \tau<\infty$, then
\begin{equation}
\frac{1}{\meas (B_2(x))}\leqslant \int_1^2\frac{s}{\meas (B_s(x))}\d s \leqslant F_x(1) \leqslant \tau<\infty,
\end{equation}
thus 
\begin{equation}\label{lower 1}
0<\frac{1}{2^N\tau}\leqslant \meas (B_1(x)) \leqslant \frac{\meas (B_s(x))}{s^N},\quad \forall s \leqslant 1.
\end{equation}
This observation plays a role at the beginning of the proof of the following.
\begin{theorem}[Compactness of non-parabolic $\RCD(0, N)$ spaces]\label{cpt non-para}
Let $(X_i, \dist_i, \meas_i, x_i)$ be a sequence of pointed non-parabolic $\RCD(0, N)$ spaces with
\begin{equation}\label{upper bd meas}
\sup_i\meas_i(B_1(x_i))<\infty
\end{equation}
and
\begin{equation}\label{upper F}
\sup_iF_{x_i}(1)<\infty.
\end{equation}
Then after passing to a subsequence, $(X_i, \dist_i, \meas_i, x_i)$ pmGH converge to a pointed $\RCD(0, N)$ space $(X, \dist, \meas, x)$ with the lower semicontinuity of $F_{x_i}$ in the sense that
\begin{equation}\label{lower f}
\liminf_{i \to \infty}F_{x_i}(r_i)\geqslant F_x(r),\quad  \text{$\forall r_i \to r$ in $(0, \infty)$.}
\end{equation} 
In particular $(X, \dist, \meas)$ is non-parabolic.
\end{theorem}
\begin{proof}

	Note that (\ref{upper F}) gives a uniform positive lower (upper, respectively) bound on $\meas_i(B_1(x_i))$ because of (\ref{lower 1}).
        Thus, thanks to the compactness of $\RCD$ spaces with respect to the pmGH topology (see for instance  \cite[Theorem 6.11]{AmbrosioGigliSavare13}, \cite[Theorem 5.3.22]{ErbarKuwadaSturm}, \cite[Theorem 7.2]{GigliMondinoSavare13}, \cite[Theorem 5.19]{LottVillani}, and \cite[Theorem 4.20]{Sturm06}), after
	 passing to a subsequence, $(X_i,\dist_i,\meas_i,x_i)$ pmGH converge to a pointed $\RCD(0, N)$ space $(X,\dist,\meas,x)$. 
	Observe that for all $r \leqslant s$, we have
	\begin{equation}
	\int_{r_i}^s\frac r{\meas_i(B_t(x_i))}\ \d t\ra\int_r^s\frac{t}{\meas(B_{t}(x))}\ \d t,
	\end{equation}
	which implies 
	\begin{equation}
	\int_r^s\frac{t}{\meas(B_{t}(x))}\ \d t= \lim_{i \to \infty}\int_{r_i}^s\frac t{\meas_i(B_t(x_i))}\ \d t \leqslant \liminf_{i \to \infty}F_{x_i}(r_i).
	\end{equation}
	Letting $s\ra\infty$, we have (\ref{lower f}).
\end{proof}
Note that (\ref{upper bd meas}) is satisfied if 
\begin{equation}\label{upper density}
\sup_i\nu_{x_i}<\infty.
\end{equation}
Because if $\nu_x \leqslant \nu<\infty$, then  
\begin{equation}\label{lower 2}
\frac{\meas (B_s(x))}{s^N}\le \nu_x \le \nu,\quad \forall s\leqslant 1
\end{equation}
by the Bishop-Gromov inequality.
Compare the following theorem with \cite[Proposition 2.3]{BrueDengSemola}.
\begin{theorem}[Convergence of Green functions]\label{green convergence}
Let us consider a pmGH convergent sequence of pointed non-parabolic $\RCD(0, N)$ spaces 
\begin{equation}\label{mghnew}
(X_i, \dist_i, \meas_i, x_i) \stackrel{\mathrm{pmGH}}{\to} (X, \dist, \meas, x).
\end{equation}
Then the following conditions are equivalent.
\begin{enumerate}
\item  The functions $f_i(s):=\frac{s}{\meas_i(B_s(x_i))}$ converge in $L^1([1, \infty), \haus^1)$ to the function $f(s):=\frac{s}{\meas(B_s(x))}$ as $i \to \infty$.
\item  $F_{x_i}(1) \to F_x(1)$.
\item  For any finite $p\geqslant 1$, $G_{x_i}$ $W^{1,p}_{\mathrm{loc}}$-strongly, and locally uniformly converge to $G_x$ on $X \setminus \{x\}$. Or equivalently $\bist_{x_i}$ $W^{1,p}_{\mathrm{loc}}$-strongly, and locally uniformly converge to $\bist_x$ on $X\setminus \{x\}$, where we say that a sequence of functions $f_i:X_i\setminus \{x_i\} \to \mathbb{R}$ locally uniformly converge to a function $f:X \setminus \{x\} \to \mathbb{R}$ if under fixing isometric embeddings $\iota_i:X_i \hookrightarrow Y, \iota:X \hookrightarrow Y$ into a common proper metric space $(Y, \dist_Y)$ realizing (\ref{mghnew}),   for any compact subset $A \subset X \setminus \{x\}$ and any $\epsilon \in (0, 1)$, there exist $\delta \in (0, 1)$ and $i_0 \in \mathbb{N}$ such that $|f_i(z_i)-f(z)|<\epsilon$ holds for all $z_i \in X_i$ and $z \in A$ whenever $\dist_Y(\iota_i(z_i), \iota(z))<\delta$ and $i \ge i_0$.

\end{enumerate}
\end{theorem}
\begin{proof}
The key point is:
\begin{itemize}
\item it is proved in \cite[Theorem 3.3]{AmbrosioHondaTewodrose} that 
	\begin{equation}\label{convheatkernel}
	p_{X_i}(y_i,z_i,t_i)\to p_X(y,z,t)
	\end{equation}
	holds for all convergent sequences of $t_i \to t$ in $(0, \infty)$ and of $y_i, z_i \in X_i \to y, z \in X$, respectively.
\end{itemize}
Based on the above, let us start giving the proof.
The implication from (1) to (2) is trivial. Assume that (2) holds. Thanks to (\ref{convheatkernel}) and  (\ref{eq:gaussian}) with the assumption, we know $\int_r^Rp_i(x_i, y_i,t)\d t \to \int_r^Rp(x, y, t)\d t$ and thus $F_{x_i}(R) \to F_x(R)$ for any $R \geqslant 1$. In particular for any $0<\eps<1$ there exists $R\geqslant 1$ such that $F_{x_i}(R) + F_x(R)<\eps$ for any $i$. On the other hand, for any fixed convergent sequence  $y_i \in X_i$ to $y \in X$ with $x \neq y$, by (\ref{eq:gaussian}), we know that there exists $0<r<1$ such that 
\begin{equation}
\int_0^rp_i(x_i, y_i, t)\d t+\int_0^rp(x, y, t)\d t<\eps.
\end{equation}
The observation above allows us to conclude the pointwise convergence $G_{x_i}(y_i) \to G_x(y)$.
Then the locally uniform convergence comes from this with a locally uniform Lipschitz bound (\ref{eq:BS.estimate.H}). Moreover since $G_{x_i}$ is harmonic on $X_i\setminus \{x_i\}$, it follows from the stability of Laplacian, \cite[Theorem 4.4]{AmbrosioHonda2}, that the $W^{1,2}_{\mathrm{loc}}$-strong convergence of the Green functions holds. Finally the improvement to the $W^{1,p}_{\mathrm{loc}}$-strong convergence is justified by combining this with (\ref{eq:BS.estimate.H}) (see also \cite{H15}). Thus we have (3).

Finally let us prove the remaining implication from (3) to (1). 
Thanks to Corollary \ref{coro:vanish.at.infinity}, for any $0<\eps<1$  there exists $R\geqslant 1$ such that $G_x(y)<\eps$ for any $y \in X \setminus B_{R}(x)$. Fix $y \in X \setminus B_{2R}(x)$ and take $y_i \in X_i$ converging to $y$.
Then our assumption allows us to conclude $G_{x_i}(y_i)<2\eps$ for any sufficiently large $i$.
Thus by (\ref{eq:BS.estimate.F}), we have 
\begin{equation}\label{eessss66s6s6s6xx}
F_{x_i}(R) \leqslant F_{x_i}(\dist_i(x_i, y_i))\leqslant C(N) G_{x_i}(y_i)\leqslant C(N)\eps.
\end{equation}
On the other hand, as discussed above, we can prove that $f_i$ converges in $L^1([1, r), \haus^1)$ to $f$ for any finite $r>1$.
This with (\ref{eessss66s6s6s6xx}) implies (1) because $\eps$ is arbitrary.
\end{proof}
Compared with Theorem \ref{cpt non-para}, it is natural to ask whether the second condition above can be replaced by a weaker one; $\sup_iF_{x_i}(1)$, or not.
However this improvement is impossible by observing a simple example discussed in subsection \ref{green sharp}. In this sense Theorem \ref{cpt non-para} is sharp.

Let us give corollaries of Theorem \ref{green convergence}. See also \cite[Corollary 2.4]{BrueDengSemola}.
\begin{corollary}\label{green asympt 0}
We have
\begin{equation}
\label{eq:2}
\lim_{y\ra x}\frac{G(x,y)}{\dist(x,y)^{2-N}}=\frac{1}{N(N-2)\nu_x}
\end{equation} 
and 
\begin{equation}
\label{eq:222}
\lim_{\dist(x,y) \to \infty}\frac{G(x,y)}{\dist(x,y)^{2-N}}=\frac{1}{N(N-2)V_X}.
\end{equation} 
\end{corollary}
\begin{proof}
We prove only (\ref{eq:2}) via a blow-up argument because the proof of (\ref{eq:222}) is similar via a blow-down argument, where the case when $\nu_x=\infty$ or $V_X=0$ directly follows from Lemma \ref{lem:asymptotic.F.and.H} with (\ref{eq:BS.estimate.F}). 

Take a convergent sequence $y_i \in X \setminus \{x\} \to x$, let $r_i:=\dist(x, y_i)$ and consider rescaled $\RCD(0, N)$ spaces;
\begin{equation}
(X_i, \dist_i, \meas_i, x_i):=\left( X, \frac{1}{r_i}\dist, \frac{1}{r_i^N}\meas, x\right).
\end{equation}
Then since
\begin{equation}\label{s0s9s8hshsbsbn}
F_{x_i}(1)=\frac{F_x(r_i)}{r_i^{2-N}} \to \frac{1}{(N-2)\nu_x}, \quad \nu_{x_i}=\nu_x,
\end{equation}
thanks to Theorem \ref{cpt non-para}, after passing to a subsequence, $(X_i, \dist_i, \meas_i, x_i)$ pmGH converge to a tangent cone $(W, \dist_W, \meas_W, w)$ at $x$, which is isomorphic to the non-parabolic $N$-metric measure cone over an $\RCD(N-2, N-1)$ space with the finite volume density $\nu_w=\nu_x$ (see (\ref{nu equiv})). By Proposition \ref{explicit green} we know $F_w(1)=\frac{1}{(N-2)\nu_w} (=\lim_{i \to \infty}F_{x_i}(1)$ by (\ref{s0s9s8hshsbsbn})) and 
\begin{equation}
G^W(w,z)=\frac{1}{N(N-2)\nu_w} \dist_W(w, z)^{2-N}, \quad \forall z \in W \setminus \{w\}.
\end{equation}
After passing to a subsequence again, we find the limit point $z$ of $y_i \in X_i$ (thus $\dist_W(w, z)=1$).
Then Theorem \ref{green convergence} shows
\begin{equation}\label{asympt FH}
\frac{G(x, y_i)}{r_i^{2-N}}=G^{X_i}(x_i, y_i) \to G^Y(w, z)=  \frac{1}{N(N-2)\nu_w}=\frac{1}{N(N-2)\nu_x}
\end{equation}
which completes the proof because $y_i$ is arbitrary.
\end{proof}
The next corollary gives an equi-convergent result on $\bist$.
Note that this corollary can be improved later under adding a uniform upper bound on the $N$-volume density. See Corollary \ref{improvement cor}.
\begin{corollary}\label{quntitative green}
For all $N >2$, $0<\eps<1$, $0<r<R<\infty$, $v>0$, $1 \leqslant p<\infty$ and $\phi \in L^1([1, \infty), \haus^1)$ there exists $\delta=\delta(N,\eps,  r, R, v, p, \phi)>0$ such that if 
two pointed non-parabolic $\RCD(0, N)$ spaces $(X_i, \dist_i, \meas_i, x_i) (i=1,2)$ satisfy $\meas_i(B_1(x_i)) \leqslant v$,
\begin{equation}
\frac{s}{\meas_i(B_s(x_i))} \leqslant \phi(s),\quad \text{for $\haus^1$-a.e. $s \in [1, \infty)$},
\end{equation}
and 
\begin{equation}\label{pmghclose1}
\dist_{\mathrm{pmGH}}\left((X_1, \dist_1, \meas_1, x_1), (X_2, \dist_2, \meas_2, x_2)\right)<\delta,
\end{equation}
then 
\begin{equation}
\left| \bist_{x_1}(y_1)-\bist_{x_2}(y_2)\right|+\left|\intav_{B_s(y_1)}|\nabla \bist_{x_1}|^p\d \meas_1 - \intav_{B_s(y_2)}|\nabla \bist_{x_2}|^p\d \meas_2\right|<\eps
\end{equation}
for all $\delta r\leqslant s\leqslant (1-\delta)r$ and $y_i \in B_R(x_i) \setminus B_r(x_i)$ satisfying that $y_1$ $\delta$-close to $y_2$ with respect to (\ref{pmghclose1}).
\end{corollary}
\begin{proof}
In order to simplify our arguments below, we give a proof only in the case when $p=2$ because the general case is similar after replacing $W^{1,2}$-convergence by $W^{1, p_i}$-convergence for a convergent sequence $p_i \to p$.

The proof is done by a strandard contradiction argument based on the compactness of $\RCD$ spaces with respect to the pmGH convergence. Namely if the assertion is not satisfied, then there exist sequences of;
\begin{itemize}
\item pointed non-parabolic $\RCD(0, N)$ spaces $(X_{j, i}, \dist_{j, i}, \meas_{j, i}, x_{j, i})$ with $\meas_{j,i}(B_1(x_{j, i}))\leqslant v$,
\begin{equation}\label{asi8s8shsbsbsus}
f_{j, i}(s):=\frac{s}{\meas_{j, i}(B_s(x_{j, i}))}\leqslant \phi(s),\quad \text{for $\haus^1$-a.e. $s \in [1, \infty)$}
\end{equation}
and
\begin{equation}
\dist_{\mathrm{pmGH}}\left((X_{1, i}, \dist_{1, i}, \meas_{1, i}, x_{1, i}), (X_{2, i}, \dist_{2, i}, \meas_{2, i}, x_{2, i})\right) \to 0;
\end{equation}
\item points $y_{j, i} \in B_R(x_{j, i}) \setminus B_r(x_{j, i})$ satisfying that $y_{1, i}$ is $\eps_i$-close to $y_{2, i}$ for some $\eps_i \to 0^+$ and that 
\begin{equation}\label{inf}
\inf_i\left( \left| \bist_{x_{1, i}}(y_{1, i})-\bist_{x_{2, i}}(y_{2, i})\right|+\left|\intav_{B_s(y_{1, i})}|\nabla \bist_{x_{1,i}}|^2\d \meas_{1, i} - \intav_{B_s(y_{2, i})}|\nabla \bist_{x_{2, i}}|^2\d \meas_{2, i}\right|\right)>0.
\end{equation}
\end{itemize}
Theorem \ref{cpt non-para} shows that after passing to a subsequence,  $(X_{j, i}, \dist_{j, i}, \meas_{j, i}, x_{j, i})$ pmGH-converge to a pointed non-parabolic $\RCD(0, N)$ space $(X, \dist, \meas, x)$.
With no loss of generality we can assume that $y_{j, i}$ converge to a point $y \in \overline{B}_R(x) \setminus B_r(x)$. Moreover the dominated convergence theorem with (\ref{asi8s8shsbsbsus}) yields that $f_{j,i}(s)$ $L^1$-strongly converge to $f(s):=\frac{s}{\meas (B_s(x))}$ in $L^1([1, \infty), \haus^1)$.
Thus Theorem \ref{green convergence} allows us to conclude
\begin{align}
&\left| \bist_{x_{1, i}}(y_{1, i})-\bist_{x_{2, i}}(y_{2, i})\right|+\left|\intav_{B_s(y_{1, i})}|\nabla \bist_{x_{1,i}}|^2\d \meas_{1, i} - \intav_{B_s(y_{2, i})}|\nabla \bist_{x_{2, i}}|^2\d \meas_{2, i}\right| \nonumber \\
&\to \left| \bist_{x}(y)-\bist_{x}(y)\right|+\left|\intav_{B_s(y)}|\nabla \bist_{x}|^2\d \meas - \intav_{B_s(y)}|\nabla \bist_{x}|^2\d \meas\right|=0
\end{align}
which contradicts (\ref{inf}).
\end{proof}
\subsection{Canonical representative of $|\nabla \bist_x|$ and drifted Laplace operator $\mathscr L$}
Throughout this subsection we  continue to argue under the same assumptions as in the previous subsection, namely we fix a pointed non-parabolic $\RCD(0, N)$ space $(X, \dist, \meas, x)$. A main result of this subsection is the following.
	\begin{theorem}[Canonical pointwise representative of $|\nabla \bist_x|$]\label{upp reg}
	The limit 
	\begin{equation}
	\lim_{r \to 0^+}\intav_{B_r(z)}|\nabla \bist_x|^2\d \meas \in [0, \infty)
	\end{equation}
	exists for any $z \in X \setminus \{x\}$. Denoting by
	$|\nabla \bist_x|^*(z)$ the square root of the limit, we have the following.
	\begin{enumerate}
	\item $|\nabla \bist_x|^*$ is upper semicontinuous.
	\item Any point $z \in X \setminus \{x\}$ is a Lebesgue point of $|\nabla \bist_x|^*$;
	\begin{equation}
	\intav_{B_r(z)}\left||\nabla \bist_x|^*-|\nabla \bist_x|^*(z)\right|\d \meas \to 0.
	\end{equation}
	\item We see that 
	\begin{equation}\label{unique}
	|\nabla \bist_x|^*(z)=\limsup_{y\ra z} |\nabla \bist_x|^*(y),\quad \forall z\in X \setminus \{x\}
	\end{equation}
	\end{enumerate} 
	\end{theorem}
	It is worth mentioning that in the proof of the theorem above, we immediately show
	\begin{equation}\label{exp rep}
	|\nabla \bist_x|^*(z)=\lim_{r \to 0^+}\esssup{y\in B_r(z)}|\nabla \bist_x|(y), \quad \forall z \in X \setminus \{x\}.
	\end{equation}
	Actually Theorem \ref{upp reg} with (\ref{exp rep}) is a direct consequence of the subharmonicity of $|\nabla \bist_x|^2G_x$ stated in Proposition \ref{prop:subha} and general results on PI spaces in \cite[Section 8.5]{BjornBjorn}.
	Thus 
	in the rest of this subsection, we focus on introducing the subharmonicity and its drifted ones which play important roles later.
	\begin{remark}
	Let us recall the following well-known fact; if $x$ is a Lebesgue point of a locally bounded function $f$ defined on an open subset $U$ of a PI space $(X, \dist, \meas)$, then
	\begin{equation}
	\intav_{B_r(x)}\psi\left( \phi \circ f-\phi (f(x))\right)\d \meas \to 0
	\end{equation}
	for all $\phi, \psi \in C(\mathbb{R})$ with $\psi(0)=0$. In particular $x$ is also a Lebesgue point of $\phi \circ f$.
	In order to prove the theorem above, we will apply this fact as $f=|\nabla \bist_x|^2$, $\phi(t)=\sqrt{|t|}$ and $\psi(t)=|t|$ with the following arguments.
	\end{remark}
	
	Consider the following drifted Laplace operator $\mathscr L$ by
	\begin{equation}\label{drift lap}
	\mathscr L:=\Delta +2\<\nabla\log G_x,\nabla \cdot \>.
	\end{equation}
	See Definition \ref{llap} for the precise definition.
	It should be emphasized that Colding studied $\mathscr L$ in \cite{C12} deeply in the smooth framework 
	in order to prove the pointwise rigidity result, (2) of Theorem \ref{thm:C12} (see \cite{CM} for applications).
	
	In the sequel, we follow his arguments, but extra delicate treatments on $\mathscr L$ are necessary in our setting because of lack of the smoothness.
	Firstly let us estimate the drifted term of (\ref{drift lap}) as follows.
	\begin{proposition}\label{prpogreengr}
		We have 
		\begin{equation}\label{Ggrad}
		|\nabla \log G_x|(y)\leqslant \frac{C(N)}{\dist (x, y)}, \quad \text{for $\meas$-a.e. $y \in X \setminus \{x\}$.}
		\end{equation}
		In particular 
		\begin{equation}\label{Ggrad2}
		\|\nabla \log G_x\|_{L^{\infty}(X \setminus B_r(x))} \leqslant \frac{C(N)}{r}, \quad \forall r>0.
		\end{equation}
		\end{proposition}
	\begin{proof}
	By Theorem \ref{thm:estimate.Green.function.in.BS},
			\begin{equation}\label{grad green log}
		|\nabla \log G_x|=\frac{|\nabla G_x|}{G_x}\leqslant C(N)\frac{H_x(\dist(x,\cdot))}{F_x(\dist(x,\cdot))}.
		\end{equation}
		On the other hand by definition we have 
		\begin{equation}
		F_x(s)=\int_s^{\infty}\frac{t}{\meas (B_t(x))}\d t \geqslant \int_s^{\infty}\frac{s}{\meas (B_t(x))}\d t =sH_x(s). 
		\end{equation}
		Combining this with (\ref{grad green log}) completes the proof.
	\end{proof}
	Let us recall the \textit{sub/super harmonicity} of a function $f$ on an open subset $\Omega$ of $X$. We say that $f$ is \textit{sub (or super, respectively) harmonic} on $\Omega$ if $f \in W^{1,2}_{\mathrm{loc}}(\Omega, \dist, \meas)$ with 
	\begin{equation}\label{eq:18}
		\int_\Omega -\<\nabla u,\nabla \phi\>\d\meas \geqslant 0,\quad \text{(or $\leqslant 0$, respectively)}
		\end{equation}
	for any $\phi\in \left(\mathrm{Lip}_c\right)_+(\Omega, \dist)$. It directly follows that $f$ is sub (or super, respectively) harmoninc on $\Omega$ if $f \in D_{\mathrm{loc}}(\Delta, \Omega)$ with $\Delta f \geqslant 0$ (or $\Delta f \leqslant 0$, respectively). See also \cite{PZZ}.
	 
	Based on this observation, we are now in a position to define the $\mathscr L$-operator precisely as follows.
	\begin{definition}[$\mathscr L$-operator and $\mathscr L$-sub/super harmonicity]\label{llap}
	Let $\Omega$  be an open subset in $X \setminus \{x\}$.
	\begin{enumerate}
	\item{($\mathscr L$-operator)} For  $u \in D_{\mathrm{loc}}(\Delta,\Omega)$, let
	\begin{equation}
	\mathscr Lu:=\Delta u+2\<\nabla\log G_x,\nabla u\> \in L^2_{\mathrm{loc}}(\Omega, \meas).
	\end{equation}
	\item{($\mathscr L$-sub/super harmonicity)} A function $u \in W^{1,2}_{\mathrm{loc}}(\Omega, \dist, \meas)$ is said to be $\mathscr L$-\textit{sub} (or \textit{super}, respectively) \textit{harmonic} on $\Omega$ if 
	\begin{equation}\label{eq:18}
		\int_\Omega -\<\nabla u,\nabla \phi\>\d\meas_{G_x}\geqslant 0,\quad \text{(or $\leqslant 0$, respectively)}
		\end{equation}
		for any $\phi\in \left(\mathrm{Lip}_c\right)_+(\Omega, \dist)$, where $\meas_{G_x}$ is the weighted Borel measure on $X$ defined by
		\begin{equation}
		\meas_{G_x}(A):=\int_AG_x^2\d \meas.
		\end{equation}
	\end{enumerate} 
	\end{definition}
	\begin{remark}
	The $\mathscr L$-operator can be defined as a measure valued one; for any $u \in D(\mathbf\Delta, \Omega)$, define
	\begin{equation}
	\mathscr{L} u := \mathbf\Delta u +2\langle \nabla \log G_x, \nabla u\rangle \d \meas,
	\end{equation}
	where $D(\mathbf\Delta, \Omega)$ is the domain of the measure valued Laplacian, see \cite{Gigli, GP2} for the detail.
	Then, even in the measure valued case, $\mathscr L$-sub/super harmonicity are also well-defined, and weak/strong maximum principles are justified. See also \cite{GR, GV}. Although we avoid to use the measure valued Laplacian/$\mathscr L$-operator for simplicity in our presentation,  however, for our main target in the sequel, $|\nabla \bist_x|^2$, the measure valued $\mathscr L$-operator, $\mathscr{L}|\nabla \bist_x|^2$ is well-defined.
	
	In connection with this, it is easy to see that $u$ is $\mathscr L$-sub (or $\mathscr L$-super, respectively) harmonic on $\Omega$ if and only if for any  $\phi\in \left(\mathrm{Lip}_c\right)_+(\Omega, \dist)$,
	\begin{equation}\label{0s9s8ssbsbsnsn}
	\int_{\Omega}-\<\nabla \phi, \nabla u\>+2\phi\<\nabla \log G, \nabla u\>\d \meas \geqslant 0,\quad \text{(or $\leqslant 0$, respectively).} 
	\end{equation}
	This observation will be a starting point in Section \ref{sec:3}.
	\end{remark}
	
	Let us introduce a standard integration-by-parts formula for $\meas_{G_x}$.
	\begin{proposition}
	Let $\Omega$  be an open subset in $X \setminus \{x\}$ and let $u \in D_{\mathrm{loc}}(\Delta,\Omega)$. Then 
	\begin{equation}
	\int_{\Omega}\langle \nabla \phi, \nabla u\rangle \d \meas_{G_x} = -\int_{\Omega}\phi \cdot \mathscr L u \d \meas_{G_x}
	\end{equation}
	for any $\phi \in \mathrm{Lip}_c(\Omega, \dist)$. In particular $u$ is $\mathscr L$-sub (or super, respectively) harmonic  on $\Omega$ if $\mathscr L u \geqslant 0$ (or $\mathscr L u \leqslant 0$, respectively).
	\end{proposition}
	\begin{proof}
	The proof is quite standard;
	\begin{align}
	\int_{\Omega}\langle \nabla \phi, \nabla u\rangle \d \meas_{G_x} &=\int_{\Omega}\langle \nabla (G_x^2\phi), \nabla u\rangle \d \meas -\int_{\Omega}\phi \langle \nabla G_x^2, \nabla u\rangle \d \meas \nonumber \\
	&= -\int_{\Omega}G_x^2\phi\Delta u\d \meas-\int_{\Omega}2\phi G_x\langle \nabla G_x,\nabla u\rangle \d \meas \nonumber \\
	&=-\int_{\Omega} \phi \cdot \mathscr L u \d \meas_{G_x}.
	\end{align}
	\end{proof}
	We are now in a position to prove a main result in this subsection, recall that the subharmonicity of $|\nabla \bist_x|^2G_x$ with results in \cite[Section 8.5]{BjornBjorn} implies Theorem \ref{upp reg} (we can find the corresponding regularity results for the $\mathscr L$-operator in Section \ref{sec:3}).
	\begin{proposition}[Subharmonicity of gradient of $\bist_x$]\label{prop:subha}
	We see that $|\nabla \bist_x|^2G_x$ is subharmonic on $X \setminus \{x\}$ and that $|\nabla \bist_x|^2$ is $\mathscr L$-subharmonic on $X \setminus \{x\}$.
	\end{proposition}
	\begin{proof}
	First of all, we claim that $|\nabla \bist_x|^2\in W^{1,2}_{\mathrm{loc}}(X\setminus\{x\}, \dist, \meas)$. For any compact set $K\subset X\setminus\{x\}$, we can take a good cut-off function $\eta\in (\Lip_c)_+(X, \dist)\cap D(\Delta)$ such that $\eta\equiv 1$ in $K$, $\supp\eta\subset X\setminus\{x\}$ and $|\nabla \eta|+|\Delta \eta|<C$ (see \cite[Lemma 3.2]{MondinoNaber} for such an existence). Letting $h:=\eta \bist_x$, then \cite[Propositions 3.3.18 and 3.3.22]{Gigli} shows that $|\nabla h|^2\in W^{1,2}(X\setminus\{x\}, \dist, \meas)$, which implies $|\nabla \bist_x|^2\in W^{1,2}_{\mathrm{loc}}(X\setminus\{x\}, \dist, \meas)$ because $K$ is arbitrary.
	
	Next recalling
	\begin{equation}
	-\int_{X\setminus \{x\}}\langle \nabla \psi, \nabla (fh)\rangle \d \meas=\int_{X\setminus \{x\}}\psi \left(h\Delta f+2\langle \nabla f, \nabla h\rangle \right)\d \meas-\int_{X\setminus \{x\}}\langle \nabla(\psi f), \nabla h\rangle \d \meas
	\end{equation}
	for all $\psi \in \mathrm{Lip}_c(X\setminus \{x\}, \dist)$, $f \in D_{\mathrm{loc}}(\Delta, X\setminus \{x\})$ and $h \in W^{1,2}_{\mathrm{loc}}(X\setminus \{x\}, \dist, \meas)$, we apply this as $f=\bist_x^2, h=|\nabla \bist_x|^2$ to get 
	\begin{align}\label{1002}
	&-\int_{X\setminus \{x\}}\langle \nabla \psi, \nabla (\bist_x^2|\nabla \bist_x|^2)\rangle \d \meas \nonumber \\
	&=\int_{X\setminus \{x\}}\psi \left(|\nabla \bist_x|^2\Delta \bist_x^2+2\langle \nabla \bist_x^2, \nabla |\nabla \bist_x|^2\rangle \right)\d \meas-\int_{X\setminus \{x\}}\langle \nabla(\psi \bist_x^2), \nabla |\nabla \bist_x|^2\rangle \d \meas \nonumber \\
	&=\int_{X\setminus \{x\}}\psi \left(2N|\nabla \bist_x|^4+2\langle \nabla \bist_x^2, \nabla |\nabla \bist_x|^2\rangle \right)\d \meas-\int_{X\setminus \{x\}}\langle \nabla(\psi \bist_x^2), \nabla |\nabla \bist_x|^2\rangle \d \meas, 
	\end{align}
	where we used (\ref{eq:6}). On the other hand, since $\bist_x^2|\nabla \bist_x|^2=\frac{|\nabla \bist_x^2|^2}{4}$, the Bochner inequality allows us to estimate the left-hand-side above as follows;
	\begin{align}\label{1001}
	-\int_{X\setminus \{x\}}\langle \nabla \psi, \nabla (\bist_x^2|\nabla \bist_x|^2)\rangle \d \meas &=-\frac{1}{4}\int_{X\setminus \{x\}}\langle \nabla \psi, \nabla |\nabla \bist_x^2|^2\rangle \d \meas \nonumber \\
	&\geqslant\frac{1}{2}\int_{X\setminus \{x\}}\psi \left(\frac{(\Delta \bist_x^2)^2}{N}+\langle \nabla \Delta \bist_x^2, \nabla \bist_x^2\rangle \right)\d \meas \nonumber \\
	&=\int_{X\setminus \{x\}}\psi \left(2N|\nabla \bist_x|^4+N\langle \nabla |\nabla \bist_x|^2, \nabla \bist_x^2\rangle \right)\d \meas.
	\end{align}
	where we used (\ref{eq:6}) again.
	Therefore it follows from (\ref{1002}) and (\ref{1001}) that
	\begin{align}\label{cocl}
	\int_{X\setminus \{x\}}(2-N)\psi \langle \nabla \bist_x^2, \nabla |\nabla \bist_x|^2\rangle \d \meas -\int_{X\setminus \{x\}}\langle \nabla (\psi \bist_x^2), \nabla |\nabla \bist_x|^2\rangle \d \meas \geqslant 0.
	\end{align}
	
	Let us prove that this inequality (\ref{cocl}) implies the conclusions. 
	Actually as done in (\ref{1002}) and (\ref{1001}), it follows from Leibniz' rule that\footnote{This is also justified by using the measure valued $\mathscr L$-operator because 
	\begin{align}
	\mathscr L |\nabla \bist_x|^2&=\mathbf\Delta |\nabla \bist_x|^2+2\<\nabla \log G, \nabla |\nabla \bist_x|^2\>\d \meas \nonumber \\
	&=\mathbf\Delta |\nabla \bist_x|^2+(2-N)\bist_x^{-2}\<\nabla \bist_x^2, \nabla |\nabla \bist_x|^2\>\d \meas.
	\end{align}}  
	\begin{align}
	&-\int_{X\setminus \{x\}}\<\nabla \psi, \nabla |\nabla \bist_x|^2\>\d \meas_{G_x}\nonumber \\
	&=\int_{X\setminus \{x\}}\left(-\<\nabla(\phi \bist_x^2), \nabla |\nabla \bist_x|^2\>+(2-N)\phi\<\nabla \bist_x^2,\nabla |\nabla \bist_x|^2\>\right)\d \meas \geqslant 0
	\end{align}
	holds, where $\phi=\bist_x^{-2}G_x^2\psi$, which proves the $\mathscr L$-subharmonicity of $|\nabla \bist_x|^2$ on $X\setminus \{x\}$.

	Similarly we have\footnote{This is also justified by using the measure valued Laplacian because 
	\begin{align}
	\mathbf\Delta (|\nabla \bist_x|^2G_x)&=G\mathbf\Delta |\nabla \bist_x|^2 +2\<\nabla G, \nabla |\nabla \bist_x|^2\> \nonumber\\
	&=\bist_x^{2-N}\mathbf\Delta |\nabla \bist_x|^2 +(2-N)\bist_x^{-N}\<\nabla \bist_x^2, \nabla |\nabla \bist_x|^2\> \nonumber \\
	&=\bist_x^{-N}\left( \bist_x^2 \mathbf\Delta |\nabla \bist_x|^2 +(2-N)\<\nabla \bist_x^2, \nabla |\nabla \bist_x|^2\>\right).
	\end{align}} 
	\begin{align}
	&-\int_{X\setminus \{x\}}\left\langle \nabla \left(|\nabla \bist_x|^2 G_x\right), \nabla \psi \right\rangle \d \meas \nonumber \\
	&=\int_{X\setminus \{x\}}\left(-\<\nabla(\phi \bist_x^2), \nabla |\nabla \bist_x|^2\>+(2-N)\phi\<\nabla \bist_x^2,\nabla |\nabla \bist_x|^2\>\right)\d \meas  \geqslant 0,
	\end{align}
	where $\phi=\psi \bist_x^{-N}$, which proves the subharmonicity of $|\nabla \bist_x|^2G_x$ on $X \setminus \{x\}$.

	\end{proof}
	In the sequel, we will use the simplified notation $|\nabla \bist_x|=|\nabla \bist_x|^*$ for the simplicity on our presentations.  In Section \ref{sec:3}, we will also provide the fundamental properties on the $\mathscr L$-operator with the proofs, including the strong maximum principle, which comes from the general theory on PI spaces. They will play important roles in the sequel.
\subsection{Sharp gradient estimate on $\bist_x$}
Fix a pointed non-parabolic $\RCD(0, N)$ space $(X, \dist, \meas, x)$ and assume that the $N$-volume density $\nu_x$ at $x$ is finite;
\begin{equation}\label{7shsnst}
\nu_x=\lim_{r\to 0^+}\frac{\meas(B_r(x))}{r^N}<\infty.
\end{equation} 
Note that this condition does not imply that the essential dimension is equal to $N$ because of an example; $([0,\infty), \dist_{\mathrm{Euc}}, r^{N-1}\d r)$ is an $\RCD(0, N)$ space with the finite $N$-volume density at the origin.

The main result of this section is the following. Recall $\mathscr C_N=(N(N-2))^{\frac{1}{N-2}}$.
\begin{theorem}[Sharp gradient estimate of $\bist_x$]
	\label{thm:sharp.gradient.estimate}
We have
\begin{equation}\label{sharp grad bd}
|\nabla \bist_x|(y)\leqslant \mathscr C_N\nu_x^\frac{1}{N-2}, \quad  \forall y \in X \setminus \{x\}.
\end{equation}
\end{theorem}
\begin{proof}
        The proof is devided into several steps as follows.

	\textbf{Step 1}.
	Let us prove that $|\nabla \bist_x|\in L^\infty(X,\meas)$. By calculus rules
	\begin{equation}
	\label{eq:4}
	|\nabla \bist_x|=\frac{1}{N-2}G_x^{\frac{N-1}{2-N}}|\nabla G_x|.
	\end{equation}
	
	Let 
	\begin{equation}
	C_{\mathrm{opt}, x}:=\lim_{r\ra 0}\ \esssup{y\in B_r(x)}\ |\nabla \bist_x|(y)  \left(= \lim_{r\ra 0}\sup_{y\in B_r(x)}\ |\nabla \bist_x|(y)\right).
	\end{equation}
	Plugging in \eqref{eq:2}, \eqref{eq:BS.estimate.H} and (\ref{eq:asymptotic.H}) we see that 
		\begin{align}
		\label{eq:13}
		C_{\mathrm{opt}, x}\leqslant  C(N) \nu_x^\frac{1}{N-2}.
		\end{align}
	Thus we can take $r$ sufficiently small such that \begin{equation}\label{eq:14}
	|\nabla \bist_x|(y)\leqslant C(N)\nu_x^\frac{1}{N-2}, \quad \meas\text{-a.e.}\ y\in B_r(x)\setminus\{x\}.
	\end{equation}

	On the other hand, combining \eqref{eq:4} with Theorem \ref{thm:estimate.Green.function.in.BS} yields that for $\meas$-a.e. $y\in X\setminus \{x\}$, we have
	\begin{equation}
	\label{eq:12}
	|\nabla \bist_x|(y)\leqslant C(N)F_x(\dist(x,y))^\frac{N-1}{2-N}H_x(\dist(x,y)).
	\end{equation}
	Note that it holds that $sH_x(s)\leqslant F_x(s)$ for any $s\in[1,\infty)$. Thus, choosing $R$ sufficiently large, then for any $y\in X\setminus B_R(x)$, we have
	\begin{equation}\label{0s0siijj}
	|\nabla \bist_x|(y)\leqslant C(N)\frac{H_x(\dist(x,y))}{F_x(\dist(x,y))^{\frac{N-1}{N-2}}}\leqslant C(N)V_X+1
	\end{equation} because of (\ref{limit asympt F}).
	Moreover it follows from the continuity of right-hand-side of \eqref{eq:12} that $|\nabla \bist_x|$ is bounded $\meas$-a.e. in $\overline{B_R(x)}\setminus B_r(x)$, which is a compact subset. Therefore $|\nabla \bist_x|\in L^\infty(X,\meas)$. 
	
	\textbf{Step 2}.	
	Let us prove that $|\nabla \bist_x|(y) \le C_{\mathrm{opt}, x}$ for any $y \in X \setminus \{x\}$.
	
   Fix $0<r<R$ and choose arbitrarily $\eps>0$. By Corollary \ref{coro:vanish.at.infinity}, we can find $R_0>R$ sufficiently large such that
	\begin{equation}
	G_x(y)< \eps, \quad \forall y\in X\setminus B_{R_0}(x)
	\end{equation}
	Take $0<r_0<r$ with
	\begin{equation}\label{eq:37}
	|\nabla \bist_x|^2(y)<C_{\mathrm{opt}, x}^2+\eps, \quad \forall y\in \overline{B}_{r_0}(x) \setminus \{x\}.
	\end{equation}
	Let $L:=\big\||\nabla \bist_x|\big\|_{L^\infty(X,\meas)}$ and
	 set
	\begin{equation}\label{eq:16}
	u_x:=|\nabla \bist_x|^2G_x-(C_{\mathrm{opt}, x}^2+\eps)G_x-L^2\eps,
	\end{equation}
	which is upper semicontinuous and subharmonic on $X\setminus\{x\}$ because of Proposition \ref{prop:subha}.
Let $\Omega:=B_{R_0}(x)\setminus \overline{B_{r_0}(x)}$. Applying the weak maximum principle for upper semicontinuous subharmonic functions \cite[Proposition 1.15]{GV}, we see that
	\begin{equation}
	\sup_{\Omega}u_x=\sup_{\partial\Omega}u_x\leqslant 0
	\end{equation}
	which proves 
	\begin{equation}
	|\nabla \bist_x|^2(y)\leqslant C_{\mathrm{opt}, x}^2+\frac{L^2\eps}{G_x(y)}+\eps, \quad \forall y \in \Omega.
	\end{equation}
	This observation allows us to conclude that $|\nabla \bist_x| \leqslant C_{\mathrm{opt}, x}$ for $\meas$-a.e. after letting $\eps \to 0^+$ under fixing $r, R$. 
	Thus by Theorem \ref{upp reg} we know that $|\nabla \bist_x|(y) \leqslant C_{\mathrm{opt}, x}$ for any $y \in X \setminus \{x\}$.  
	
	\textbf{Step 3}.
	We claim that $C_{\mathrm{opt}, x}=\mathscr C_N\nu_x^{\frac{1}{N-2}}$, where this completes the proof. Thanks to (\ref{nu equiv}), Proposition \ref{explicit green} and Theorem \ref{green convergence}, we know that for any $0<\delta <1$
	\begin{equation}\label{00s9s9s8}
	\intav_{B_r(x)\setminus B_{\delta r}(x)}\left| |\nabla \bist_x|^2-\mathscr C_N^2\nu_x^{\frac{2}{N-2}}\right|\d \meas \to 0, \quad \text{as $r \to 0^+$.}
	\end{equation}
	Take $y_i \to x$ satisfying $|\nabla \bist_x|(y_i) \to C_{\mathrm{opt}, x}$ 
	and let $r_i:=\dist(x, y_i)$ and consider rescaled spaces;
	\begin{equation}
	(X_i, \dist_i, \meas_i):=\left(X, \frac{1}{r_i}\dist, \frac{1}{r_i^N}\meas \right).
	\end{equation}
	Applying the weak Harnack inequality for $\mathscr L$-superharmonic functions, Proposition \ref{prop:weak.harnack.super}, to a lower semicontinuous $\mathscr L$-superharmonic function $C_{\mathrm{opt}, x}^2-|\nabla \bist_{x}^{X_i}|^2 \geqslant 0$,  we have
	\begin{align}\label{s8shshsbsyststst}
	\left(\intav_{B_{\frac{1}{4}}^{\dist_i}(y_i)}\left|C_{\mathrm{opt}, x}^2-|\nabla \bist_{x}^{X_i}|^2\right|^p\d \meas_i \right)^{1/p}&\leqslant C(N) \left(C_{\mathrm{opt}, x}^2-|\nabla \bist_x^{X_i}|^2(y_i)\right) \nonumber \\
	&=C(N) \left(C_{\mathrm{opt}, x}^2-|\nabla \bist_x|^2(y_i)\right) \to 0,
	\end{align}
	where $p=p(N)>0$.
	Thus recalling $|\nabla \bist_x| \in L^{\infty}(X\setminus \{x\}, \meas)$, it holds that 
	\begin{equation}
	\intav_{B_{\frac{1}{4}}^{\dist_i}(y_i)}\left|C_{\mathrm{opt}, x}^2-|\nabla \bist_{x}^{X_i}|^2\right|\d \meas_i \to 0,
	\end{equation}
	namely
	\begin{equation}\label{9s9s9s9}
	\intav_{B_{\frac{r_i}{4}}(y_i)}\left| C_{\mathrm{opt}, x}^2-|\nabla \bist_{x}|^2\right|\d \meas \to 0.
	\end{equation}
	On the other hand, (\ref{00s9s9s8}) implies 
	\begin{equation}\label{9s9s9s99}
	\intav_{B_{\frac{r_i}{4}}(y_i)}\left| \mathscr C_N^2\nu_x^{\frac{2}{N-2}}-|\nabla \bist_{x}|^2\right|\d \meas \to 0.
	\end{equation}
	Thus by (\ref{9s9s9s9}) and (\ref{9s9s9s99}) we have  $C_{\mathrm{opt}, x}=\mathscr C_N\nu_x^{\frac{1}{N-2}}$.
	\end{proof}
	We provide direct consequences of Theorem \ref{thm:sharp.gradient.estimate}. Firstly we improve Theorems \ref{green convergence} and \ref{upp reg} removing the singular base point.
\begin{corollary}[Improvement of the convergence of $\bist$]\label{impr}
Let us consider a pmGH convergent sequence of pointed non-parabolic $\RCD(0, N)$ spaces 
\begin{equation}
(X_i, \dist_i, \meas_i, x_i) \stackrel{\mathrm{pmGH}}{\to} (X, \dist, \meas, x)
\end{equation}
with (\ref{upper density}) and $F_{x_i}(1) \to F_x(1)$. Then $\bist_{x_i}$ $W^{1,p}_{\mathrm{loc}}$-strongly, and locally uniformly converge to $\bist_x$ on $X$ for any $p<\infty$. 
\end{corollary}
\begin{proof}
Note
\begin{equation}\label{small ball est}
\frac{1}{\meas_i (B_1(x_i))}\int_{B_{\eps}(x_i)}|\nabla \bist_{x_i}|^p\d \meas_i \leqslant \frac{\meas_i(B_{\eps}(x_i))}{\meas_i(B_1(x_i))}\mathscr C_N^p \cdot \sup_j \nu_{x_j}^{\frac{p}{N-2}} \to \frac{\meas (B_{\eps}(x))}{\meas (B_1(x))}\mathscr C_N^p \cdot \sup_j\nu_{x_j}^{\frac{p}{N-2}}.
\end{equation}
Since the right-hand-side of (\ref{small ball est}) is small if $\eps$ is small, combining this with Theorem \ref{green convergence} completes the proof.
\end{proof}
\begin{corollary}\label{cor best lip}
Let $\bist_x(x):=0$. Then $\bist_x$ is $\mathscr C_N\nu_x^{\frac{1}{N-2}}$-Lipschitz with the Lipschitz constant $\mathscr C_N\nu_x^{\frac{1}{N-2}}$. 
Moreover letting
\begin{equation}
|\nabla \bist_x|(x):=\mathscr C_N\nu_x^{\frac{1}{N-2}},
\end{equation}
we have the same conclusions as in Theorem \ref{upp reg} with (\ref{exp rep}) even for the base point $z=x$. 
More strongly, we have for any $p<\infty$.
\begin{equation}\label{strongco}
\intav_{B_r(x)}\left|\nabla \left(\bist_x-\mathscr C_N\nu_x^{\frac{1}{N-2}}\dist_x\right)\right|^p\d \meas \to 0,\quad r \to 0^+.
\end{equation}
In particular, 
 $|\nabla \bist_x|(y_i) \to \mathscr C_N\nu_x^{\frac{1}{N-2}}$ for some convergent sequence $y_i \to x$.
\end{corollary}
\begin{proof}
Thanks to (\ref{eq:BS.estimate.F}) and (\ref{eq:asymptotic.F}), putting $\bist_x(x):=0$ gives a unique continuous extention of $\bist_x$ on $X$.
Since $\{x\}$ is null with respect to the $2$-Sobolev capacity because of the finiteness of $\nu_x$ (see \cite{BjornBjorn}), we know $\bist_x \in W^{1, 2}_{\mathrm{loc}}(X, \dist, \meas)$. Then the first statement comes from arguments in the last step in the proof of Theorem \ref{thm:sharp.gradient.estimate} with the (local) Sobolev-to-Lipschitz property and Corollary \ref{impr}. Moreover Corollary \ref{impr} with Proposition \ref{explicit green} allows us to obtain (\ref{strongco}).
\end{proof}
As the final application of Theorem \ref{thm:sharp.gradient.estimate}, we determine the small scale asymptotics of the gradient of the Green function.
\begin{corollary}\label{cor asympt FH}
We have 
\begin{equation}
\label{eq:3}
\lim_{r \to 0^+}  \sup_{y \in B_r(x)}\frac{|\nabla G_x|(y)}{\dist(x,y)^{1-N}}=\frac{1}{N\nu_x}.
\end{equation}
\end{corollary}
\begin{proof}
Since (\ref{eq:2}),  (\ref{eq:4}) and Corollary \ref{cor best lip} yield (under a suitable limit $y \to x$)
\begin{align}
\frac{|\nabla G_x|(y)}{\dist (x,y)^{1-N}}&=(N-2)\left(\dist (x,y)G_x(y)^{\frac{1}{N-2}}\right)^{N-1}|\nabla \bist_x|(y)\nonumber \\
& \to (N-2) \left( \frac{1}{\mathscr C_N\nu_x^{\frac{1}{N-2}}}\right)^{N-1}\mathscr C_N\nu_x^{\frac{1}{N-2}} \nonumber \\
&=\frac{1}{N\nu_x}
\end{align}
we conclude.
\end{proof}
Next we provide an asymptotic formula as $y \to \infty$.
\begin{corollary}[Sharp gradient asymptotics]\label{thm:almost.rigidity3}
For any $p<\infty$ we have 
\begin{equation}\label{9s9s9sbbb}
\intav_{B_r(x)}\left|\nabla \left(\bist_x-\mathscr C_NV_X^{\frac{1}{N-2}}\dist_x\right)\right|^p\d \meas \to 0,\quad r \to \infty,
\end{equation}
therefore
\begin{equation}\label{0as9au0rasuhs}
\intav_{B_r(x)}\left||\nabla \bist_x|^2-\mathscr C_N^2V_X^{\frac{2}{N-2}}\right| \d \meas \to 0,\quad r \to \infty.
\end{equation}
In particular 
			\begin{equation}\label{s8s8s8s8shhs}
			|\nabla \bist_x|(y_i)\to \mathscr C_NV_X^{\frac{1}{N-2}},
			\end{equation}
			equivalently
			\begin{equation}
\label{eq:3sssss}
\frac{|\nabla G_x|(y_i)}{\dist(x,y_i)^{1-N}} \to \frac{1}{NV_X}
\end{equation}
			 holds for some sequence $y_i \in X$ with $\dist(x, y_i) \to \infty$.
\end{corollary}
\begin{proof}
Firstly we discuss the case when $V_X=0$. Then we can follow the same arguments as in the proof of \cite[Theorem 2.12]{C12}. Namely we can estimate as $\dist(x, y) \to \infty$,
\begin{align}
|\nabla \bist_x|(y) &=\frac{G_x^{\frac{N-1}{2-N}}}{N-2}|\nabla G_x|(y) \nonumber \\
&\leqslant \frac{G_x(y)^{\frac{N-1}{2-N}}}{N-2}\cdot C(N) \cdot G_x(y) \cdot \frac{1}{\dist(x,y)}\nonumber \\
&\leqslant  C(N)\left( \frac{F_x(\dist(x, y))}{\dist(x,y)^{2-N}}\right)^{\frac{1}{2-N}} \to 0,
\end{align}
where we used the gradient estimates on positive harmonic functions obtained in \cite[Theorem 1.2]{Jiang} in the first inequality above and we also used (\ref{limit asympt F}) and (\ref{eq:BS.estimate.F}) in the last inequality and in the limit.
Thus we obtain the conclusion in this case.

Next we consider the case when $V_X>0$.
The first statement, (\ref{9s9s9sbbb}),  is a direct consequence of Corollary \ref{impr} and (2) of Corollary \ref{volconemetr}. The remaining one (\ref{s8s8s8s8shhs}) (or (\ref{eq:3sssss})) follows from an argument similar to the proof of Corollary \ref{cor asympt FH}.
\end{proof}
Based on Corollary \ref{impr}, we can prove the following whose proof is the same to that of Corollary \ref{quntitative green}. Thus we omit the proof.
\begin{corollary}\label{improvement cor}
For all $N >2$, $0<\eps<1$, $0<\tau<1$, $v>0$, $1 \leqslant p<\infty$ and $\phi \in L^1([1, \infty), \haus^1)$ there exists $\delta=\delta(N,\eps, \tau,  v, p, \phi)>0$ such that if 
two pointed non-parabolic $\RCD(0, N)$ spaces $(X_i, \dist_i, \meas_i, x_i) (i=1,2)$ satisfy $\nu_{x_i} \leqslant v<\infty$,
\begin{equation}
\frac{s}{\meas_i(B_s(x_i))}\leqslant \phi(s),\quad \text{for $\haus^1$-a.e. $s \in [1, \infty)$}
\end{equation}
and 
\begin{equation}\label{pmghclose}
\dist_{\mathrm{pmGH}}\left((X_1, \dist_1, \meas_1, x_1), (X_2, \dist_2, \meas_2, x_2)\right)<\delta,
\end{equation}
then 
\begin{equation}
\left| \bist_{x_1}(y_1)-\bist_{x_2}(y_2)\right|+\left|\intav_{B_s(y_1)}|\nabla \bist_{x_1}|^p\d \meas_1 - \intav_{B_s(y_2)}|\nabla \bist_{x_2}|^p\d \meas_2\right|<\eps
\end{equation}
for all $\tau \leqslant s\leqslant \tau^{-1}$ and $y_i \in B_{\tau^{-1}}(x_i)$ satisfying that $y_1$ $\delta$-close to $y_2$ with respect to (\ref{pmghclose}).
\end{corollary}
Finally let us end this section by giving the following corollary which generalizes \cite[Theorem 3.5]{FMP} to the RCD setting. This corollary is pointed out by the reviewer. We thank the reviewer.
\begin{corollary}\label{newcorollary}
The function $f_x:(0, \infty) \to [0, \infty)$ defined by
\begin{equation}
f_x(t):=\sup_{\{\bist_x=t\}}|\nabla \bist_x|
\end{equation}
is monotone non-increasing, where recall that we assume $\nu_x<\infty$. Moreover we have 
\begin{equation}\label{zero limit}
\lim_{t \to 0^+}f_x(t)  =\mathscr C_N\nu_x^{\frac{1}{N-2}}
\end{equation}
and
\begin{equation}\label{nfty limit}
\lim_{ t \to \infty}f_x(t) = \mathscr C_NV_X^{\frac{1}{N-2}}.
\end{equation}
\end{corollary}
\begin{proof}
For the first statement, we just follow the proof of \cite[Theorem 3.5]{FMP}. Namely our goal is to prove for any $t>0$
\begin{equation}\label{newresu}
|\nabla \bist_x|  \leqslant \sup_{\{\bist_x=t\}}|\nabla \bist_x|,\quad \text{on $\{\bist_x \ge t\}$}
\end{equation}
because then for any $s \le t$ we have
\begin{equation}
f_x(s)= \sup_{\{\bist_x \ge s\}}|\nabla \bist_x| \geqslant \sup_{\{\bist_x \ge t\}}|\nabla \bist_x| = f_x(t)
\end{equation}
which completes the proof of the first statement. Thus let us focus on the proof of (\ref{newresu}). Note that $G_x$ is $\mathcal{L}$-harmonic on $X \setminus \{x\}$. Fixing a sufficiently large $T>t$, consider an $\mathcal{L}$-subharmonic function $\phi$ on $X \setminus \{x\}$ defined by
\begin{equation}
\phi(y):=|\nabla \bist_x|^2(y)-f_x(t)-\frac{\mathscr C_N^2\nu_x^{\frac{2}{N-2}}}{T^{N-2}G_x(y)}.
\end{equation}
It is trivial that $\phi \le 0$ holds on $\{ \bist_x =t \} \cup \{\bist_x= T\}$.  Thus the strong maximum principle, Proposition \ref{thm:strong.max.princi.}, shows $\phi \le 0$ on $\{t \leqslant \bist_x \leqslant T\}$. Then letting $T \to \infty$ proves (\ref{newresu}). 

On the other hand (\ref{zero limit}) is a direct consequence of Theorem \ref{thm:sharp.gradient.estimate} and Corollary \ref{cor best lip}. Thus let us focus on the proof of (\ref{nfty limit}). The proof is divided into the following $2$ cases.

\textbf{Case 1}: $V_X=0$.

By the same argument as in (\ref{0s0siijj}), we see that for any $\epsilon>0$, there exists $R>1$ such that 
\begin{equation}
|\nabla \bist_x|(y)  \leqslant C(N)V_X+\epsilon=\epsilon
\end{equation}
holds for any $y \in X \setminus B_R(x)$, namely we have $\limsup_{\dist_x(y) \to \infty}|\nabla \bist_x|(y) \leqslant C(N) V_X=0$ which proves (\ref{zero limit}) in this case.

\textbf{Case 2}: $V_X>0$.

In this case the proof is similar to \textbf{Step 3} of the proof of Theorem \ref{thm:sharp.gradient.estimate} with (\ref{newresu}). Let $C_{\mathrm{opt}, \infty}:=\lim_{t \to \infty}f_x(t)$. Take $x_t \in \{\bist_x=t\}$ with $|\nabla \bist_x|(x_t)=f_x(t)$, and let  $0\epsilon_t:=f_x(t)^2-C_{\mathrm{opt}, \infty}^2\geqslant 0$.
Note that $C_{\mathrm{opt}, \infty}^2-|\nabla \bist_x|^2 +\eps_t \geqslant 0$ holds on $\{\bist_x\geqslant t\}$. 

On the other hand, thanks to (\ref{eq:222}), for any sufficiently large $t>1$, we have $B_{\delta t}(x_t) \subset \{\bist_x\geqslant \frac{t}{2}\}$, where $\delta=\delta(N, V_X)$ is a positive constant depending only on $N, V_X$. In particular
\begin{equation}
C_{\mathrm{opt}, \infty}^2-|\nabla \bist_x|^2 +\eps_{\frac{t}{2}} \geqslant 0,\quad \text{on $B_{\delta t}(x_t)$.}
\end{equation}
Then, after a rescaling $t^{-1}\dist, \meas (B_t(x))^{-1}\meas$, applying the weak Harnack inequality for $\mathcal{L}$-superharmonic function, Proposition \ref{prop:weak.harnack.super}, to $C_{\mathrm{opt}, \infty}^2-|\nabla \bist_x|^2 +\eps_{\frac{t}{2}} \geqslant 0$, we have as $t \to \infty$
\begin{equation}\label{9s9s9s910}
	\left(\intav_{B_{\frac{t}{4}}(x_{t})}\left| C_{\mathrm{opt}, \infty}^2-|\nabla \bist_{x}|^2+\eps_{\frac{t}{2}}\right|^p\d \meas\right)^{1/p} \leqslant C(N)  \left(C_{\mathrm{opt}, \infty}^2-|\nabla \bist_{x}|^2(x_t)+\eps_{\frac{t}{2}}\right) \to 0,
	\end{equation}
	for some $p=p(N)>0$, namely
	\begin{equation}
	\intav_{B_{\frac{t}{4}}(x_{t})}\left| C_{\mathrm{opt}, \infty}^2-|\nabla \bist_{x}|^2\right|\d \meas \to 0.
	\end{equation}
	Combining this with (\ref{0as9au0rasuhs}) yields $C_{\mathrm{opt}, \infty}=\mathscr C_NV_X^{\frac{1}{N-2}}$ which completes the proof of (\ref{nfty limit}). 
\end{proof}

\section{Rigidity to $N$-metric measure cone}\label{sec:5}
We are now in a position to prove the main results. Fix a finite $N >2$.
\subsection{Rigidity}

Let us prove the desired rigidity result based on Theorems \ref{upp reg} and \ref{thm:sharp.gradient.estimate} (see (\ref{optimalc}) for the definition of $\mathscr C_N$).
\begin{theorem}[Rigidity to $N$-metric measure cone]\label{thm:rigidity}
	Let $(X,\dist,\meas, x)$ be a pointed non-parabolic $\RCD(0,N)$ space with the finite $N$-volume density $\nu_x<\infty$. 
	If there exists a point $y \in X\setminus\{x\}$ such that
	\begin{equation}
	|\nabla \bist_x|(y)=\mathscr C_N\nu_x^\frac{1}{N-2},
	\end{equation}
	then $(X,\dist,\meas,x)$ is isomorphic to the $N$-metric measure cone with the pole over an $\RCD(N-2, N-1)$ space.
\end{theorem}
\begin{proof}
First of all, let us prove that $X \setminus \{x\}$ is connected. Let $U$ be a connected component of $X\setminus \{x\}$. As mentioned in the beginning of the proof of Corollary \ref{cor best lip}, since $\{x\}$ has a null $2$-capacity, the indicator funcion $\chi_U$ of $U$ is in $H^{1,2}_{\mathrm{loc}}(X, \dist, \meas)$ with $|\nabla \chi_U|=0$. Thus the (local) Sobolev-to-Lipschitz property shows $\chi_U= 1$ for $\meas$-a.e., which implies $X\setminus \{x\}=U$. Thus $X\setminus \{x\}$ is connected.

Then applying  the strong maximum principle, Proposition \ref{thm:strong.max.princi.}, for $|\nabla \bist_x|^2$ yields $|\nabla \bist_x|\equiv \mathscr C_N\nu_x^\frac{1}{N-2}$ in $X \setminus \{x\}$.  
Letting $u:={\bist_x^2}/ (2\mathscr C_N^2\nu_x^\frac{2}{N-2})$, we have $|\nabla \sqrt{2u}|^2=1$ and thus, by \eqref{eq:6}, $\Delta u=N$ for $\meas$-a.e. 
These observations allow us to apply \cite[Theorem 5.1]{GV} to get the conclusion.
\end{proof}

In particular, when restricted to points with Euclidean tangent spaces, Theorem \ref{thm:rigidity} implies an interesting corollary; compare with \cite[Theorem 3.1]{C12}.
\begin{corollary}\label{cor:rigidity euclidean}
We have the following.
        \begin{enumerate}
        \item 
	Let $(X,\dist,\meas)$ be a non-parabolic $\RCD(0,N)$ space for some integer $N \geqslant 3$ with the finite $N$-volume density $\nu_x<\infty$ at an $N$-regular point $x$. If there exists $z\in X \setminus \{x\}$ such that $|\nabla \bist_x|(z)=\mathscr C_N\nu_x^{\frac{1}{N-2}}$ holds, then $(X,\dist, \meas)$ is isometric to $(\mathbb{R}^N, \dist_{\mathbb{R}^N}, c\haus^N)$ for some positive constant $c >0$.
	\item If a non-parabolic non-collapsed $\RCD(0, N)$ space $(X, \dist, \haus^N)$  for some integer $N \geqslant 3$ with the finite $N$-volume density $\nu_x<\infty$ at a point $x \in X$ satisfies
	$|\nabla \bist_x|(z) \ge \mathscr C_N \omega_N^{\frac{1}{N-2}}$ for some $z \in X \setminus \{x\}$,
	then $(X, \dist, \haus^N)$ is isometric to $(\mathbb{R}^N, \dist_{\mathbb{R}^N}, \haus^N)$.
	\end{enumerate}
	\end{corollary}
	\begin{proof}
	Let us prove (1).
	Theorem \ref{thm:rigidity} yields that $(X, \dist, \meas)$ is isomorphic to an $N$-metric measure cone over an $\RCD(N-2, N-1)$ space. In particular it must be isomorphic to a tangent cone at $x$. Thus we conclude.
	
	Next let us prove (2).
	The sharp gradient estimate, Theorem \ref{thm:sharp.gradient.estimate}, yields 
	\begin{equation}
	\lim_{r \to 0^+}\frac{\haus^N(B_r(x))}{\omega_Nr^N} \ge 1,
	\end{equation}
	thus $x$ is a $N$-regular point because of \cite[Corollary 1.7]{DePhillippisGigli} (recall (\ref{bishop ine})). Then the conclusion follows from the first statement (1).
	\end{proof}
	

\subsection{Almost rigidity}
Finally, let us prove the following almost rigidity theorem of $\bist_x$:
		\begin{theorem}[Almost rigidity]\label{thm:almost.rigidity}
		For all $N>2$, $0<\eps<1$, $0<r<R<\infty$, $\nu>0$, $1\leqslant p<\infty$ and $\phi \in L^1([1, \infty), \haus^1)$ there exists $\delta:=\delta(N, \eps, r, R, \nu, p, \phi)>0$ such that if 
			 a pointed non-parabolic $\RCD(0,N)$ space $(X,\dist,\meas,x)$ satisfies $\nu_x\leqslant \nu$, 
			\begin{equation}
			\frac{s}{\meas(B_s(x))}\leqslant \phi(s), \quad \text{for $\haus^1$-a.e. $s \in [1, \infty)$}
			\end{equation}
			and
			\begin{equation}
			|\nabla \bist_x|(z)\geqslant \mathscr C_N\nu_x^\frac{1}{N-2}-\delta
			\end{equation}
			for some $z\in \overline{B}_R(x)\setminus B_r(x)$, 
			then we have
			\begin{equation}\label{asah7sasash}
\left\| \bist_x-\mathscr C_N\nu_x^{\frac{1}{N-2}}\dist_x\right\|_{L^{\infty}(B_R(x))}+\left(\intav_{B_R(x)}\left| \nabla \left(\bist_x -\mathscr C_N\nu_x^{\frac{1}{N-2}}\dist_x\right)\right|^p\d \meas\right)^{1/p} \leqslant \eps
\end{equation}
and there exists an $\RCD(N-2,N-1)$ space $(Y,\dist_Y,\meas_Y)$ such that
			\begin{equation}
			\dist_{\mathrm{pmGH}}\left((X,\dist,\meas,x),(C(Y),\dist_{C(Y)},\meas_{C(Y)},O_Y)\right)+|\nu_x-\nu_{O_Y}|<\eps
			\end{equation}
			
		\end{theorem}
\begin{proof}
As in the proof of Corollary \ref{quntitative green}, it is enough to consider the case when $p=2$.
Then the proof is done by a contradiction, thus
assume the conclusion fails. Then there exist sequences of positive numbers $\delta_i\to 0^+$ and of pointed $\RCD(0,N)$ spaces $(X_i,\dist_i,\meas_i,x_i)$ such that $\nu_{x_i} \leqslant \nu$, $\frac{s}{\meas_i(B_s(x_i))} \leqslant \phi (s)$ for $\haus^1$-a.e. $s \in [1, \infty)$, that 
	\begin{equation}
	|\nabla_i \bist_{x_i}|(z_i)\geqslant \mathscr C_N\nu_{x_i}^\frac{1}{N-2}-\delta_i,\quad \text{for some $z_i \in \overline{B}_R(x_i) \setminus B_r(x_i)$}
	\end{equation}
	and that  for any $\RCD(N-2,N-1)$ space $(Y,\dist_Y,\meas_Y)$, we have
\begin{equation}\label{eq:44}
\inf_i \left(\dist_{\mathrm{pmGH}}\left((X_i,\dist_i,\meas_i,x_i),(C(Y),\dist_{C(Y)},\meas_{C(Y)},O_Y)\right)+|\nu_{x_i}-\nu_{O_Y}|\right)>0
\end{equation}	
or 
\begin{equation}\label{asah7sasashassssusu}
\inf_i\left(\left\| \bist_{x_i}-\mathscr C_N\nu_{x_i}^{\frac{1}{N-2}}\dist_{x_i}\right\|_{L^{\infty}(B_R(x_i))}+\left(\intav_{B_R(x_i)}\left| \nabla \left(\bist_{x_i} -\mathscr C_N\nu_{x_i}^{\frac{1}{N-2}}\dist_{x_i}\right)\right|^p\d \meas\right)^{1/p} \right)>0.
\end{equation}
	Theorem \ref{cpt non-para} allows us to conclude that after passing to a subsequence, $(X_i,\dist_i,\meas_i,x_i)$ pmGH converge to a pointed non-parabolic $\RCD(0, N)$ space $(X,\dist,\meas,x)$, and $\nu_{x_i} \to \mu$ for some $\mu \in [0, \infty)$. 
	On the other hand, the lower semicontinuity of $N$-volume densities (\ref{lower semicont}) implies
	\begin{equation}\label{lower density1}
	\nu_{x}\leqslant \lim_{i\to \infty}\nu_{x_i}=\mu\leqslant \nu<\infty.
	\end{equation}
	Consider $\mathscr L$-superharmonic lower semicontinuous functions on $X_i\setminus \{x_i\}$;
		\begin{equation}
		u_i:=\mathscr C_N^2\nu_{x_i}^\frac 2{N-2}-|\nabla_i \bist_{x_i}|^2 \geqslant 0.
		\end{equation} Fix $0<s<\frac{r}{4}$.
		 Applying the weak Harnack inequality, Proposition \ref{prop:weak.harnack.super}, to the nonegatively valued $\mathscr L$-superharmonic function $u_i$, we have 
		\begin{equation}\label{eq:61}
		\left(\intav_{B_{2s}(z_i)}u_i^p\ \d\meas_i\right)^\frac 1 p\leqslant C\inf_{B_s(z_i)} u_i\leqslant C u_i(z_i)\leqslant C\delta_i.
		\end{equation}
		where $p=p(N, s)>0$.
	Thus recalling that $u_i$ is uniformly bounded, we have 
	\begin{equation}\label{app max}
	\intav_{B_{2s}(z_i)}\left|\mathscr C_N^2\nu_{x_i}^\frac 2{N-2}-|\nabla_i \bist_{x_i}|^2\right|\ \d\meas_i\ra 0.
	\end{equation}
	Note that with no loss of generality we can assume that $z_i$ converge to a point $z \in \overline{B}_R(x)\setminus B_r(x)$. Moreover the dominated convergence theorem yields the convergence of $\frac{s}{\meas_i(B_s(x_i))}$ to $\frac{s}{\meas (B_s(x))}$ in $L^1([1, \infty), \haus^1)$.
	Thus Theorem \ref{green convergence} shows that $\bist_{x_i}$ $W^{1,2}$-strongly converge to $\bist_{x}$ on $B_{2s}(z)$. Therefore (\ref{app max}) with (\ref{lower density1}) implies
	\begin{equation}
	\intav_{B_{2s}(z)}|\nabla \bist_x|^2\d \meas = \mathscr C_N^2\mu^{\frac{2}{N-2}} \geqslant \mathscr C_N^2\nu_x^{\frac{2}{N-2}}.
	\end{equation}
	In particular Theorem \ref{thm:sharp.gradient.estimate} shows $\mu=\lim_{i\to \infty}\nu_{x_i}=\nu_x$ and
	\begin{equation}
	|\nabla b_x|=\mathscr C_N\nu_x^\frac 1 {N-2}, \quad \text{for $\meas$-a.e. in $B_{2s}(z)$.}
	\end{equation}  Thus, Theorem \ref{thm:rigidity} allows us to conclude that the limit space  is isomorphic to the $N$-metric measure cone over an $\RCD(N-2, N-1)$ space, which contradicts \eqref{eq:44} and (\ref{asah7sasashassssusu}).
\end{proof}

Next we provide an almost rigidity to a Euclidean space on a non-collapsed space.
\begin{corollary}\label{thm:almost.rigidity2}
For any integer $N\geqslant 3$, all $0<\eps<1$, $0<r<R$ and $\phi \in L^1([1, \infty), \haus^1)$ there exists $\delta:=\delta(N, \eps, r, \phi)>0$ such that if 
a pointed non-parabolic non-collapsed $\RCD(0,N)$ space $(X,\dist,\haus^N,x)$ satisfies
			\begin{equation}
			\frac{s}{\haus^N(B_s(x))}\leqslant \phi(s),\quad \text{for $\haus^1$-a.e. $s \in [1, \infty)$}
			\end{equation}
			and
			\begin{equation}
			|\nabla \bist_x|^2(y)\geqslant \mathscr C_N \omega_N^{\frac{1}{N-2}}-\delta
			\end{equation}
			for some $y \in \overline{B}_R(x) \setminus B_r(x)$,
			then $(X, \dist, \haus^N, x)$ is $\eps$- pmGH close to $(\mathbb{R}^N, \dist_{\mathbb{R}^N}, \haus^N, 0_N)$. 
\end{corollary}
\begin{proof}
The first statement is a direct consequence of (the proof of) Theorem \ref{thm:almost.rigidity} with the Bishop inequality $\nu_x \leqslant \omega_N$.		
\end{proof}
\begin{remark}\label{rem:EH}
Let us remark that the conclusion of Corollary \ref{thm:almost.rigidity2} cannot be replaced by a stronger one;
\begin{equation}\label{optimal max}
V_X\geqslant \omega_N-\eps.
\end{equation}
Find an open $N$-manifold $(M^N, g)$ with the maximal volume growth which is not isometric to $\mathbb{R}^N$ (for instance the Eguchi-Hanson metric on the cotangent bundle $T^*\mathbb{S}^2$ of $\mathbb{S}^2$ gives such an example with vanishing Ricci curvature, see \cite{EHM}).
Then since the tangent cone at infinity is not isometric to $\mathbb{R}^N$, the asymptotic $N$-volume $V_{M^N}$ is away from $\omega_N$. Fix $x \in M^N$. Take any convergent sequence $x_i \to x$  with $x_i \neq x$ and then recall
\begin{equation}\label{conv sharp}
|\nabla \bist_x|(x_i) \to \mathscr C_N \omega_N^{\frac{1}{N-2}}.
\end{equation}
Thus considering the rescaled distance $\dist_i:=\dist(x, y_i)^{-1}\dist$, the pointed Riemannian manifolds $(M^N, \dist_i, x)$ satisfies (\ref{conv sharp}) for some $y_i \in \partial B_1^{\dist_i}(x)$, but the asymptotic $N$-volume is away from $\omega_N$ because of the scale invariance of $V_{M^N}$. In particular (\ref{optimal max}) is not satisfied in this case.
\end{remark}


In connection with this remark, we prove the following.
\begin{theorem}\label{thm:almost.rigidity5}
For any integer $N \geqslant 3$, all $0<\eps<1$ and $\tau>0$ there exists $\delta=\delta(N, \eps, \tau)>0$ such that if
a pointed non-parabolic non-collapsed $\RCD(0,N)$ space  $(X,\dist,\haus^N,x)$ satisfies
			$
			V_X \geqslant \tau
			$
			and
			 
			 \begin{equation}
			 \mathscr C_N\omega_N^{\frac{1}{N-2}}-|\nabla \bist_x|(y_i)\leqslant \delta
			 \end{equation}
			 for  some sequence $y_i\in X$ with $\dist(x, y_i)\to \infty$, then
			\begin{equation}\label{vol close}
V_X\geqslant \omega_N-\eps
\end{equation}
In particular if $\eps $ is sufficiently small depending only on $N$ and $\tau$, then $X$ is homeomorphic to $\mathbb{R}^N$, moreover in addition, if $X$ is smooth, then the homeomorphism can be improved to be a diffeomorphism.
\end{theorem}
\begin{proof}
Let $r_i:=\dist(x, y_i)$ and consider a rescaled pointed non-parabolic non-collapsed $\RCD(0, N)$ space;
\begin{equation}
(X_i,  \dist_i, \haus^N_{\dist_i}, x_i):=\left( X, \frac{1}{r_i}\dist, \frac{1}{r_i^N}\haus^N_{\dist}, x\right).
\end{equation}
Recalling (\ref{rescaling F}) we have for any $r \geqslant 1$
\begin{equation}\label{00ii}
F_{x_i}^{X_i}(r)=\frac{F_x^X\left( rr_i\right)}{r_i^{2-N}} \leqslant \frac{r^{2-N}}{(N-1)\tau},
\end{equation}
where we used (\ref{ps9s8shbbbsn}) in the final inequality. Thus Corollary \ref{thm:almost.rigidity2} allows us to conclude that $(X_i,  \dist_i, \haus^N_{\dist_i}, x_i)$ is pmGH close to the $N$-dimensional Euclidean space. In particular a tangent cone at infinity is also pmGH close to the  $N$-dimensional Euclidean space. Therefore the volume convergence result, \cite[Theorem 1.3]{DePhillippisGigli}, implies (\ref{vol close}). Thus we get the first statement.
The remaining statements come from this with the same arguments as in \cite[Theorems A.1.11]{CheegerColding1}.
\end{proof}
\begin{remark}\label{final rem}
In the theorem above, in order to get the same conclusion, we cannot replace the existence of divergent points by the existence of a point which is far from $x$. The reason is the same to Remark \ref{rem:EH}.
\end{remark}

\section{Examples}\label{example sec}
In this section we see that Theorems \ref{green convergence} and \ref{thm:almost.rigidity} are sharp via simple examples. Moreover we also discuss a related sharpness and open problems.
\subsection{Sharpness I}\label{sharp obs}
In this subsection we prove that Theorem \ref{thm:almost.rigidity} is sharp, namely this cannot be improved to the case when $p=\infty$.  The following arguments also allow us to conclude that (\ref{impsosirbb}) and (\ref{elpest222}) in Theorem \ref{thm:smooth almost} are also sharp.

\textbf{Step 1}.
Consider $\mathbb{S}^2(r):=\{x \in \mathbb{R}^3 | |x|_{\mathbb{R}^3}=r\}$ for $r<1$ with the standard Riemannian metric and denote by $X=C(\mathbb{S}^2(r))$ the $3$-metric measure cone with the $3$-dimensional Hausdorff measure $\haus^3_X (=\meas_{C(\mathbb{S}^2(r))})$. Assume that $r$ is close to $1$, and take a point $y \in X$ which is close to the pole $x \in X$ with $x \neq y$. Since $(X, \dist, \haus^3, y)$ is pmGH close to the $3$-dimensional Euclidean space, Theorem \ref{thm:almost.rigidity} yields that 
\begin{equation}
\intav_{B_1(y)}\left| |\nabla \bist_y|-4\pi\right|\d \haus^3_X
\end{equation}
is small,
where the Lipschitz constant can be calculated by $(3(3-2)\omega_3)^{\frac{1}{3-2}}=4\pi$.

On the other hand, we have 
\begin{equation}\label{9shssbsrtstsg}
|\nabla \bist_y|(x)=\lim_{r \to 0^+}\left(\intav_{B_r(x)}|\nabla \bist_y|^2\d \haus^3_X\right)^{1/2}=0
\end{equation}
because of the same trick observed in \cite{DePhillippisZ}. Namely, thanks to Lemmas \ref{lem form}, \ref{rescaling fh} and the stability of the Laplacian \cite[Theorem 4.4]{AmbrosioHonda2}, under any blow-up at $x$, $\bist_x$ $W^{1,2}_{\mathrm{loc}}$-strongly converge to a linear growth harmonic function on $C(\mathbb{S}^2(r))$. Recalling that any such function must be a constant because $r<1$, we have (\ref{9shssbsrtstsg})

Then recalling the upper semicontinuity of $|\nabla \bist_y|$, we know that $|\nabla \bist_y|$ is small around $x$. In particular 
\begin{equation}
4\pi-|\nabla \bist_y| \geqslant 3\pi
\end{equation}
near $x$, thus
\begin{equation}
\||\nabla \bist_y|-4\pi\|_{L^{\infty}(B_1(y))} \geqslant 3\pi.
\end{equation}

\textbf{Step 2}.
Let 
\begin{equation}
(Z_i, \dist_i, \meas_i, z_i) \stackrel{\mathrm{pmGH}}{\to} (Z, \dist, \meas, z)
\end{equation}
 be a pmGH convergent sequence of pointed $\RCD(K, N)$ spaces and let $f_i \in L^{\infty}(B_R(x_i), \meas_i)$ $L^p$-strongly converge to $f \in L^{\infty}(B_R(x), \meas)$ for any $p<\infty$ with $\sup_i\|f_i\|_{L^{\infty}}<\infty$. Then 
 \begin{equation}\label{labsysheyrb}
 \liminf_{i\to \infty}\|f_i\|_{L^{\infty}}\geqslant \|f\|_{L^{\infty}}.
 \end{equation} Because for any $p<\infty$, since 
 \begin{equation}
 \liminf_{i\to \infty}\|f_i\|_{L^{\infty}}\geqslant \lim_{i\to \infty}\left( \intav_{B_R(z_i)}|f_i|^p\d \meas_i\right)^{1/p}=\left(\intav_{B_R(z)}|f|^p\d\meas\right)^{1/p},
 \end{equation}
 letting $p \to \infty$ completes the proof of (\ref{labsysheyrb}).
 
 \textbf{Step 3}. Let us consider $X=C(\mathbb{S}^2(r))$ again as in \textbf{Step 1}. Note that it is easy to find a sequence of manifolds of dimension $3$, $(M^3_i, g_i, \haus^3, y_i)$, pmGH-converge to $(X, \dist, \haus^3, y)$ (actually $(X, \dist, \haus^3, x)$ is the tangent cone at infinity of a $3$-dimensional complete Riemannian manifold $(M^3, g)$ with $V_{M^3}=V_X>0$). Then applying (\ref{labsysheyrb}) for $f_i=4\pi-|\nabla \bist_{y_i}|, f=4\pi-|\nabla \bist_y|$ with Corollary \ref{impr} implies
 \begin{equation}
 \liminf_{i \to \infty}\|4\pi-|\nabla \bist_{y_i}|\|_{L^{\infty}(B_1(y_i))}\geqslant \|4\pi-|\nabla \bist_y|\|_{L^{\infty}(B_1(y))} \geqslant 3\pi.
 \end{equation}
This observation shows that we cannot improve Theorem \ref{thm:almost.rigidity} to the case when $p=\infty$.
 
\subsection{Sharpness II}\label{green sharp}
In the rest two subsections \ref{green sharp} and \ref{green sharp1}, we prove that several results we obtained previously under assuming;
\begin{equation}\label{a9s8s8shsbsvsvsvsv}
\frac{s}{\meas(B_s(x))}\leqslant \phi(s)
\end{equation}
do not hold if we replace (\ref{a9s8s8shsbsvsvsvsv}) by a weaker one;
\begin{equation}\label{sisbsbs71bs0ansuer6}
F_x(1)\leqslant C<\infty
\end{equation}
based on examples discussed in the previous subsection. Namely the assumption (\ref{a9s8s8shsbsvsvsvsv}) is sharp in these results.

The first one is about Theorem \ref{green convergence}, namely
we
provide an example of pmGH convergent sequences of non-parabolic $\RCD(0,3)$ spaces whose Green functions do not converge to the limit one, though the corresponding $F_{x_i}(1)$ are bounded.

For fixed $0<r<1$,\footnote{It is not difficult to see that the following arguments are also justified even in the case when $r=1$ after choosing a suitable projection $\pi$.} let $X=C(\mathbb{S}^2(r))$ with the pole $x$, the cone distance $\dist=\dist_{C(\mathbb{S}^2(r))}$ and the $3$-dimensional Hausdorff measure $\haus^3_X$. For any $R>0$, let us denote by $X_R$ the glued space of a closed ball $\overline{B}_R(x)$ and a cylinder $\partial B_R(x) \times [0, \infty)$ along the boundary $\partial B_R(x):=\{y\in X|\dist (x, y)=R\}$. Then it is trivial that both $X$ and $X_R$ with the canonical intrinsic distance $\dist_R$ can be canonically realized as boundaries $\partial D, \partial D_R$ of closed convex subsets $D, D_R$ in $\mathbb{R}^4$ with $x=0_4$ and $D_R \subset D$, respectively. In particular $(X_R, \dist_R, \haus^3_{X_R})$ is an $\RCD(0, 3)$ space. 
Note that $(X_R, \dist_R, \haus^3_{X_R})$ is not non-parabolic.

Under the conventions above, denote by $\pi_R:X \to X_R$ the canonical projection in $\mathbb{R}^4$ (thus $\pi_R|_{\overline{B}_R(x)}=\mathrm{id}_{\overline{B}_R(x)}$). For any $0\leqslant t \leqslant 1$, define $\pi_{R, t}:X \to D$ by $\pi_{R, t}(y):=(1-t)y+t\pi_R(y)$, and put $X_{R, t}:=\pi_{R, t}(X)$ which is also the boundary of a covex closed subset in $\mathbb{R}^4$.

Choose $R_{0}>1$ with 
\begin{equation}
\int_{R_{0}}^{\infty}\frac{s}{\haus^3_X(B_s(x))}\d s<1
\end{equation}
because of the non-parabolicity of $(X, \dist, \haus^3_X)$. Fix $R \geqslant R_{0}^2$. Since
\begin{equation}
\int_R^{\infty}\frac{s}{\haus^3_{X_{R, t}}(B_s(x))}\d s \to \int_R^{\infty}\frac{s}{\haus^3_{X_R}(B_s(x))}\d s =\infty,\quad t \to 1^-
\end{equation}
because $(X_R, \dist_R, \haus^3_{X_R})$ is not non-parabolic, we can find $t_{R} \in (0, 1)$ with
\begin{equation}
\int_R^{\infty}\frac{s}{\haus^3_{X_{R, t_{R}}}(B_s(x))}\d s=1.
\end{equation}
Then let us consider a pmGH-convergent sequence of non-parabolic $\RCD(0, 3)$ spaces;
\begin{equation}
\left(X_{R, t_{R}}, \dist_{R, t_{R}}, \haus^3_{X_{R, t_{R}}}, x\right) \stackrel{\mathrm{pmGH}}{\to} \left(X, \dist, \haus^3_X, x\right),\quad R \to \infty,
\end{equation}
where $\dist_{R, t_{R}}$ denotes the canonical intrinsic distance on $X_{R, t_{R}}$.
Since
\begin{align}
F_x^{X_{R, t_{R}}}(1)&=\int_1^R\frac{s}{\haus^3_{X_{R, t_{R}}}(B_s(x))}\d s+1\nonumber \\
&=\int_1^R\frac{s}{s^3\haus^3_X(B_1(x))}\d s +1 \to F_x^X(1)+1,\quad R \to \infty,
\end{align}
Theorem \ref{green convergence} tells us that the Green functions do not converge to the limit one.
Actually we can see this directly as follows. 

Fix $y \in X\setminus \{x\}$. Then for any sufficiently large $R$, Gaussian estimates (\ref{eq:gaussian}) show
\begin{align}
G^{X_{R, t_{R}}}(x, y)&=\int_0^Rp_{X_{R, t_{R}}}(x, y, t)\d t+\int_R^{\infty}p_{X_{R, t_{R}}}(x, y, t)\d t\nonumber \\
&\geqslant \int_0^Rp_{X_{R, t_{R}}}(x, y, t)\d t +\frac{1}{C}\int_{\sqrt{R}}^{\infty}\frac{s}{\haus^3_{X_{R, t_{R}}}(B_s(x))}\d s \nonumber \\
&=\int_0^Rp_{X_{R, t_{R}}}(x, y, t)\d t +\frac{1}{C}
\end{align}
for some $C>1$.
Thus letting $R \to \infty$ with Fatou's lemma and (\ref{convheatkernel}) yields
\begin{equation}\label{9ashaoriaasiuarh}
\liminf_{R\to \infty}G^{X_{R, t_{R}}}(x, y)\geqslant \int_0^{\infty}p_X(x, y, t)\d t+\frac{1}{C}=G^X(x,y)+\frac{1}{C}.
\end{equation}

\subsection{Sharpness III}\label{green sharp1}
The final sharpness result is related to  Corollary \ref{improvement cor}.
An immediate consequence of Corollary \ref{improvement cor} states that if a pointed non-parabolic $\RCD(0, N)$ space $(X, \dist, \meas, x)$ with the finite $N$-volume density $\nu_x<\infty$ is pmGH-close to the $N$-metric measure cone over an $\RCD(N-2, N-1)$ space, then $|\nabla \bist_x|$ attains the maximum $\mathscr C_N\nu_x^{\frac{1}{N-2}}$ almostly at some point which is bounded and away from $x$, whenever (\ref{a9s8s8shsbsvsvsvsv}) holds. In the sequel we prove that (\ref{a9s8s8shsbsvsvsvsv}) cannot be replaced by (\ref{sisbsbs71bs0ansuer6}) to get the same conclusion.

To do so, under the same notations as in   subsection \ref{green sharp}, let us discuss the behavior of $\bist_x^{X_{R, t_R}}$ as $R \to \infty$. Our claim is that if $r$ is small, then there exists no sequence $y_{R_i} \in X_{R_i, t_{R_i}}$ as $R_i \to \infty$ such that $y_{R_i}$ is bounded and is away from $x$ and that 
\begin{equation}
|\nabla \bist_x^{X_{R_i, t_{R_i}}}|(y_{R_i}) \to \mathscr C _3\nu_x(=3\nu_x)
\end{equation}
The proof is done by a contradiction. If such sequence $y_{R_i}$ exists, then after passing to a subsequence, with no loss of generality we can assume that $y_{R_i}$ converge to some $y \in X\setminus \{x\}$ and that $\bist_{x}^{X_{R_i}, t_{R_i}}$ locally uniformly converge to some $\tilde \bist \in \mathrm{Lip}(X, \dist)$. Note that thanks to Lemma \ref{lem form} and the stability of the Laplacian \cite[Theorem 4.4]{AmbrosioHonda2}, we know that $\tilde \bist \in D(\Delta, X \setminus \{x\})$ and $\tilde \bist^2\in D(\Delta)$ hold with 
\begin{equation}
\Delta \tilde \bist=2\frac{|\nabla \tilde \bist|^2}{\tilde \bist},\quad \Delta \tilde \bist^2=6|\nabla \tilde \bist|^2.
\end{equation}
On the other hand, applying a weak Harnack inequality, Theorem \ref{prop:weak.harnack.super}, for $(3\nu_x)^2-|\nabla \bist_x^{X_{R_i, t_{R_i}}}|^2$ as in (\ref{s8shshsbsyststst}) proves
\begin{equation}
\int_{B_R(x) \setminus B_r(x)}\left|(3\nu_x)^2-|\nabla \bist_x^{X_{R_i, t_{R_i}}}|^2\right|\d \haus^3_{X_{R_i, t_{R_i}}} \to 0
\end{equation}
for all $0<r<R<\infty$. Thus we know 
\begin{equation}
\Delta \tilde \bist^2=6(3\nu_x)^2. 
\end{equation}
Then since $\tilde \bist^2-(3\nu_x)^2\dist_x^2$ is a harmonic function on $X$ with polynomial growth of degree at most $2$, any such function must be a constant if $r$ is small. Thus we have $\tilde \bist^2=(3\nu_x)^2\dist_x^2+d$ for some $d \in \mathbb{R}$. Since $\tilde \bist(x)=0$ by definition, we know $d=0$, namely
\begin{equation}
\tilde \bist=3\nu_x\dist_x.
\end{equation}
In paritcular we have ${\tilde \bist}^{-1}=G^X_x$ which contradicts (\ref{9ashaoriaasiuarh}). Thus the observation above allows us to conclude for all $0<r_1<r_2<\infty$
\begin{equation}
\liminf_{R\to \infty}\left(\inf_{B_{r_2}(x)\setminus B_{r_2}(x)}\left(\mathscr C_3\nu_x-|\nabla \bist_x^{X_{R, t_R}}|\right)\right)>0.
\end{equation}
\begin{remark}
We do not know whether we can replace (\ref{a9s8s8shsbsvsvsvsv}) by  (\ref{sisbsbs71bs0ansuer6}) to get the same conclusions in Theorems \ref{thm:smooth almost} and \ref{thm:1.4}.
\end{remark}
\section{Appendix; Analysis on a drifted Laplace operator $\mathscr L$}\label{sec:3}
In this appendix, we provide detailed proofs of the regularity results for $\mathscr L$-sub/super harmonic functions (recall Definition \ref{llap}), coming directly from the general theory on PI spaces. It is emphasized that these techniques can be applied to more general operators, including, of course,  our Laplacian $\Delta$. We refer \cite[Section 8.5]{BjornBjorn} as a main reference on this topic.
	
	Let $(X, \dist, \meas, x)$ be a pointed non-parabolic $\RCD(0, N)$ space.
	The first result is about the weak Harnack inequality for $\mathscr L$-subharmonic functions. This is justified by applying \cite[Theorem 8.4]{BjornBjorn} to an (incomplete) metric measure space $(B_{100r}(y), \dist, \meas_{G_x})$ because of (\ref{eq:BS.estimate.F}) and (\ref{Ggrad}). 
	\begin{proposition}[Weak Harnack inequality for $\mathscr L$-subharmonic functions]\label{lem:weak.harnack.sub22}
		Let $u$ be an $\mathscr L$-subharmonic function on a ball $B_{100r}(y)$ for some $r \leqslant 1$ with $B_{100r}(y) \subset X\setminus B_s(x)$ for some $s \leqslant 1$. 
		Then for all $k\in\R$ and $p>1$, there exists $C=C(N,s,p)>0$ such that\footnote{When we apply directly \cite[Theorem 8.4]{BjornBjorn}, then the conclusion should be written in terms of $\meas_{G_x}$ instead of $\meas$ in (\ref{eq:45}). The difference can be understood as follows.
			Recalling $\d \meas_{G_x}=G_x^2\d \meas =\exp (2\log G_x)\d \meas$, for fixed $y$ as in the statement, under assuming $\|\nabla \log G_x\|_{L^{\infty}} \le L$ on the domain, we have 
			\begin{equation}\label{7snnbbshhss}
			\frac{1}{L} \leqslant \frac{G_x(z)}{G_x(w)}\leqslant L 
			\end{equation}
			because of applying the upper gradient inequality for $\log G_x$. This inequality, (\ref{7snnbbshhss}), allows us to compare $\meas$ with $\meas_G$ quantitatively. Thus we can state the conclusion in terms of $\meas$. The same observation can be applied in the sequel.
		}
		\begin{equation}\label{eq:45}
		\esssup{B_{\frac{r}{2}}(y)}u+k\leqslant C\left(\intav_{B_r(y)}(u+k)_+^p\ \d\meas\right)^\frac 1 p.
		\end{equation}
	\end{proposition}

	\begin{proof}
			First let us assume that $p>2$. For arbitrary $l>0$, let $\hat u=\hat u_k:=(u+k)_+$, and denote by $\bar u=\bar u_{k,l}:=\min\{\hat u_k,l\}$. For $\frac 1 2 r\leqslant r_1<r_2\leqslant r$, take the cutoff function
			\begin{equation}
			\eta(z):=\min\left\{1,\frac{r_2-\dist(z,y)}{r_2-r_1}\right\}_+.
			\end{equation}
			Then we have $0\leqslant \eta\leqslant 1$, $|\nabla \eta|\leqslant \frac 1{r_2-r_1}$ in $B_{r_2}(y)$, $\eta\equiv 1$ in $B_{r_1}(y)$, and $\eta\equiv 0$ outside $B_{r_2}(y)$. For any  $\beta> p-2>0$, define
			\begin{equation}
			v=v_{k,l,\beta}:=\eta^2\bar u^\beta\hat u\in W^{1,2}_{0,+}(B_{100r}(y)).
			\end{equation}
			By direct calculation
			\begin{equation}
			\nabla v=2\eta\bar u^\beta\hat u\nabla \eta+\eta^2\bar u^\beta(\beta\nabla \bar u+\nabla \hat u).
			\end{equation}
			Substitute $\phi=v$ in \eqref{0s9s8ssbsbsnsn},
			\begin{equation}\label{eq:1113}
			\begin{aligned}
			\int_{B_{r_2}(y)}\left(2\eta\bar u^\beta\hat u\<\nabla \eta,\nabla u\>+\eta^2\bar u^\beta(\beta|\nabla \bar u|^2+|\nabla \hat u|^2)
			-
			2\eta^2\bar u^\beta\hat u\<\nabla \log G_x,\nabla u\>\right)\ \d\meas\leqslant 0.
			\end{aligned}
			\end{equation}
			Let $E_k:=\{u\geqslant -k\}$, noting that $\hat u\equiv 0$ outside $E_k$, the integrand above vanishes in $B_{r_2}(y)\setminus E_k$. Using Young's inequality, noticing that fact that $|\nabla \hat u|=|\nabla u|$ in $E_k$,
			\begin{equation}\label{eq:1112}
			\begin{aligned}
			\chi_{E_k}|2\eta\bar u^\beta\hat u\<\nabla \eta,\nabla u\>|\leqslant \chi_{E_k}\left(\frac 1 3\eta^2\bar u^\beta|\nabla \hat u|^2+3 \bar u^\beta\hat u^2|\nabla \eta|^2\right),
			\end{aligned}
			\end{equation}
			\begin{equation}\label{eq:25}
			\chi_{E_k}\left|2\eta^2\bar u^\beta\hat u\<\nabla\log G_x,\nabla u\>\right|\leqslant \chi_{E_k}\left(\frac 1 3\eta^2\bar u^{\beta}|\nabla \hat u|^2+3\eta^2\bar u^{\beta}\hat u ^2|\nabla \log G_x|^2\right).
			\end{equation}
			Combining \eqref{eq:1113}, \eqref{eq:1112} and \eqref{eq:25}, it holds that
			\begin{equation}
			\int_{B_{r_2}(y)} \eta^2\bar u^\beta(\beta|\nabla \bar u|^2+|\nabla \hat u|^2)\ \d\meas\leqslant 9\int_{B_{r_2}(y)}\bar u^\beta\hat u^2(|\nabla \eta|^2+\eta^2|\nabla \log G_x|^2)\ \d\meas.
			\end{equation}
			Set $w:=\bar u^\frac\beta 2\hat u$, noting that $\hat u\geqslant \bar u$ and $|\nabla\bar u|$ vanishes in $\{\hat u> l\}$, it is easy to check
			\begin{equation}
			|\nabla w|^2\leqslant \left(\frac\beta 2\bar u^{\frac\beta 2-1}\hat u|\nabla \bar u|+\bar u^\frac{\beta}{2}|\nabla \hat u|\right)^2\leqslant (1+\beta)\bar u^\beta\left(\beta|\nabla \bar u|^2+|\nabla \hat u|^2\right).
			\end{equation}
			Therefore,
			\begin{equation}\label{eq:31}
			\int_{B_{r_2}(y)}\eta^2|\nabla w|^2\ \d\meas\leqslant 9(1+\beta)\int_{B_{r_2}(y)}w^2(|\nabla \eta|^2+\eta^2|\nabla \log G_x|^2)\ \d\meas.
			\end{equation}
			By the Sobolev inequality
			\begin{equation}
			\begin{aligned}
			\left(\int_{B_{r_2}(y)}(\eta w)^{2\xi}\ \d\meas\right)^\frac 1\xi&\leqslant C(N)\frac{r_2^2}{\meas(B_{r_2}(y))^{\frac 2 N}}\int_{B_{r_2}(y)}|\nabla(\eta w)|^2\ \d\meas\\
			&\leqslant C(N)\frac{(1+\beta)r_2^2}{\meas(B_{r_2}(y))^{\frac 2 N}}\int_{B_{r_2}(y)}w^2(|\nabla \eta|^2+\eta^2|\nabla \log G_x|^2)\ \d\meas,
			\end{aligned}
			\end{equation}
			where $\xi=\xi(N):=\frac{N}{N-2}$. Recall that $|\nabla \log G_x|^2$ is bounded above by $C(N, s)$, by the choice of $\eta$,
			\begin{equation}\label{eq:1114}
			\|w\|_{L^{2\xi}(B_{r_1}(y),\meas)}\leqslant C(N,s)\frac{(1+\beta)^\frac 1 2}{\meas(B_{r_2}(y))^{\frac 1 N}}\frac{r_2}{r_2-r_1}\| w\|_{L^2(B_{r_2}(y),\meas)}
			\end{equation}
			For any $\gamma>2$ and $t\in[\frac 1 2r,r]$, set the quantity
			\begin{equation}
			A(\gamma,t):=\left(\int_{B_r(y)}\bar u^{\gamma-2}\hat u^2\ \d\meas\right)^\frac 1 \gamma.
			\end{equation}
			Then \eqref{eq:1114} yields that
			\begin{equation}\label{eq:1115}
			\begin{aligned}
			A(\xi(\beta+2),r_1)=&\left(\int_{B_{r_1}(y)}\bar u^{\xi(\beta+2)-2}\hat u^2\ \d\meas\right)^\frac 1 {\xi(\beta+2)}\\
			\leqslant&
			\left(\int_{B_{r_1}(y)}w^{2\xi}\ \d\meas\right)^\frac 1 {\xi(\beta+2)}\\
			\leqslant&
			\left(C(N,s)\frac{(1+\beta)^\frac 1 2}{\meas(B_{r}(y))^{\frac 1 N}}\frac{r}{r_2-r_1}\right)^\frac{2}{\beta+2}A(\beta+2,r_2),
			\end{aligned}
			\end{equation}
			where we used the following fact from Bishop-Gromov inequality:
			\begin{equation}
			\frac{r_2}{\meas(B_{r_2}(y))^\frac 1 N}\leqslant \frac{r}{\meas(B_{r}(y))^\frac 1 N}.
			\end{equation}
			Let $\beta_n=p\xi^n-2>0$, $\gamma_n:=\beta_n+2>2$, $t_n:=(2^{-1}+2^{-n-1})r$, iterating \eqref{eq:1115},
			\begin{equation}
			A(\gamma_n,t_n)\leqslant\Prod{i=1}{n}\left(2^{i+1}C(N,s)(\gamma_i-1)^\frac 1{2}\right)^\frac 2{\gamma_i} \meas(B_r(y))^{-\frac {2(\xi-\xi^{-n})} {Np(1-\xi)}}A(p, r),
			\end{equation}
			Let $n\ra \infty$, recalling by definition \begin{equation} \|\bar u\|_{L^\gamma(B_t(y),\meas)}\leqslant A(\gamma,t)\leqslant \|\hat u\|_{L^\gamma(B_t(y),\meas)},\end{equation}
			it holds that
			\begin{equation}
			\|\bar u_{k,l}\|_{L^\infty(B_\frac{r}{2}(t),\meas)}\leqslant C(N,s,p)\meas(B_r(y))^{-\frac 1 p}\|\hat u_k\|_{L^p(B_{r}(y),\meas)}.
			\end{equation}
			Since the right-handed side of the inequality is independent on $l$, we may let $l\ra\infty$, namely
			\begin{equation}
			\|\hat u_{k}\|_{L^\infty(B_\frac{r}{2}(t),\meas)}\leqslant C(N,s,p)\meas(B_r(y))^{-\frac 1 p}\|\hat u_k\|_{L^p(B_{r}(y),\meas)},
			\end{equation}
			which finishes the proof under the assumption $p>2$. 
			
			The case where $1<p\leqslant 2$ can be treated via a similar iteration process provided $\hat{u} \in L^\infty(B_r(y),\meas)$. Indeed, for any $\beta>p-1>0$, we can choose a simpler test function without truncation:
			\begin{equation}
			v_{k,\beta}:=\eta^2(u+k)_+^\beta\in W^{1,2}_{0,+}(B_{100r}(y)),
			\end{equation} 
			which can further simplify the proof. Here we omit the details.

	\end{proof}
	Next let us discuss about $\mathscr L$-superharmonic functions. The corresponding results for PI spaces can be found in \cite[Theorem 8.10]{BjornBjorn}. In order to establish a weak Harnack inequality for $\mathscr L$-superharmonic functions, we first recall the definition of bounded mean oscillating (BMO) functions on metric measure spaces and John-Nirenberg's lemma.
	\begin{definition}[BMO]\label{bmo}
		Let $(X,\dist,\meas)$ be a measure-doubling metric measure space, namely there exists a constant $C_d>0$ such that
		\begin{equation}\label{eq:101}
		\meas(B_{2r}(x))\leqslant C_d\meas(B_r(x)), \quad \forall x\in X,\ \forall r>0.
		\end{equation}
		For any open subset $\Omega\subset X$ and any function $f\in L^1_{\mathrm{loc}}(\Omega,\meas)$, we set
		\begin{equation}
		\|f\|_{\mathrm{BMO}(\Omega)}:=\sup_B\intav_B\left|f-f_B\right|\ \d\meas,
		\end{equation}
		where
		the supremum is taken over all open balls $B$ with $B\subset \Omega$. The class of BMO functions on $\Omega$ is the collection
		\begin{equation}
		\mathrm{BMO}(\Omega):=\left\{f\mathop\big|\|f\|_{\mathrm{BMO}(\Omega)}<\infty\right\}.
		\end{equation}
	\end{definition}
	\begin{remark}Note that the doubling assumption is naturally fulfilled for $\RCD(0,N)$ spaces because of the Bishop-Gromov inequality. In particular, in this case, the doubling constant $C_d$ is only dependent on $N$. \end{remark}
	John-Nirenberg's lemma we refer is stated as follows, see \cite[Theorem 3.20]{BjornBjorn}.
	\begin{theorem}[John-Nirenberg's lemma]\label{johnir} Let $(X, \dist, \meas)$ be as in Definition \ref{bmo} and let  $f \in \mathrm{BMO}(B_{5r}(x))$ for some $x\in X$ and $r>0$. Then for any $0<\eps\leqslant A:=\log 2/(4C_d^{15})$, 
		\begin{equation}
		\intav_{B_r(x)}\exp\left(\frac{\eps|f-f_{B_r(x)}|}{\|f\|_{\mathrm{BMO}(B_{5r}(x))}}\right)\ \d\meas\leqslant \frac{A+\eps}{A-\eps}.
		\end{equation}
	\end{theorem}
	We are now in a position to prove a weak Harnack inequality (see \cite[Theorem 8.10]{BjornBjorn}).
	\begin{proposition}[Weak Harnack inequality for $\mathscr L$-superharmonic functions]\label{prop:weak.harnack.super}
		Let $u$ be a non-negatively valued $\mathscr L$-superharmonic function on a ball $B_{100r}(y)$ for some $r \leqslant 1$ with $B_{100r}(y) \subset X\setminus B_s(x)$ for some $s \leqslant 1$. 
		Then there exist  $p=p(N, s)>0$ and $C=C(N, s)>1$ such that 
		\begin{equation}\label{eq:46}
		C\essinf{B_r(y)}u\geqslant\left(\intav_{B_{2r}(y)}u^p \d\meas\right)^\frac 1 p.
		\end{equation}
	\end{proposition}
	\begin{proof}
		First let us assume $u$ is bounded away from $0$. For any fixed $\eps>0$, there exists a piecewise linear function $\psi$ in the form
		\begin{equation}
		\psi(t):=\max_{1\leqslant i\leqslant k}(a_it+b_i),
		\end{equation}
		where $a_i<0$, such that 
		\begin{equation}
		\frac 1 t\leqslant \psi(t)\leqslant \frac 1 t+\eps,\quad \forall t>\eps.
		\end{equation}
		Applying Proposition \ref{lem:weak.harnack.sub22} to $\psi\circ (u+\eps)$, we obtain that for any $p>0$, 
		\begin{equation}
		\begin{aligned}
		&\frac 1{\essinf{B_r(y)}u+\eps}\leqslant \esssup{B_r(y)}\psi\circ(u+\eps)\leqslant C(N,s,p)\left(\intav_{B_{2r}(y)}(\psi\circ(u+\eps))^p\ \d\meas\right)^\frac 1 p\leqslant\\
		&C(N,s,p)\left(\intav_{B_{2r}(y)}\left(\frac 1{u+\eps}+\eps\right)^p\ \d\meas\right)^\frac 1 p\leqslant C(N,s,p)\left(\intav_{B_{2r}(y)}\left(\frac 1{u}+\eps\right)^p\ \d\meas\right)^\frac 1 p.
		\end{aligned}
		\end{equation}
		Letting $\eps\ra 0$, we obtain that 
		\begin{equation}\label{eq:102}
		\begin{aligned}
		\essinf{B_r(y)}u
		&\geqslant C(N,s,p)\left(\intav_{B_{2r}(y)}u^{-p}\ \d\meas\right)^{-\frac 1 p}\\
		&=C(N,s,p)\left(\intav_{B_{2r}(y)}u^{-p}\ \d\meas\intav_{B_{2r}(y)} u^p\ \d\meas\right)^{-\frac 1 p}\left(\intav_{B_{2r}(y)}u^p\ \d\meas\right)^\frac 1 p.
		\end{aligned}
		\end{equation}
		Now it suffices to show that there exists some $p=p(N,s)>0$ such that 
		\begin{equation}
		\intav_{B_{2r}(y)}u^{-p}\ \d\meas\intav_{B_{2r}(y)} u^p\ \d\meas\leqslant C(N,s).
		\end{equation}
		On the other hand, recalling (\ref{0s9s8ssbsbsnsn}), we can establish an analog of \eqref{eq:31} similarly as in the proof of Proposition \ref{lem:weak.harnack.sub22} for $w:=\log u$ treating the test function $v:=\eta^2$
		\begin{equation}
		\int_{B_{100r}(y)} \eta^2|\nabla w|^2\ \d\meas\leqslant C(N)\int_{B_{100r}(y)} \left(|\nabla \eta|^2+\eta^2|\nabla \log G_x|^2\right)\ \d\meas,
		\end{equation}
		where $\eta$ is any cut-off function with $\supp \eta \subset B_{100r}(y)$. 
		Thus choosing $\eta$ with $\eta \equiv 1$ on $B_{50r}(y)$, we obtain
		\begin{equation}
		\intav_{B_{50r}(y)}|\nabla w|^2\ \d\meas\leqslant \frac{ C(N, s)}{r^2}.
		\end{equation}
		Thus using the Poincar\'e inequality (\ref{poincareineq}), for any $z\in B_{10r}(y)$ and $r'<20r$, we have
		\begin{equation}
		\intav_{B_{r'}(z)}\left|w-\intav_{B_{r'}(z)}w\ \d\meas\right|\ \d\meas\leqslant C(N)r\left(\intav_{B_{r'}(z)}|\nabla w|^2\ \d\meas\right)^\frac 1  2\leqslant C(N,s),
		\end{equation}
		which implies
		\begin{equation}\label{eq:103}
		\|w\|_{\mathrm{BMO}(B_{10r}(x))}\leqslant C(N,s).
		\end{equation}
		On the other hand, for this $C(N, s)$ in (\ref{eq:103}),
		applying John-Nirenberg's lemma, Theorem \ref{johnir},  to $w$ with $\eps:=A \cdot C(N, s)/2$ (recall $A$ is taken as a dimensional constant in this setting), we have
		\begin{equation}
		\intav_{B_{2r}(y)}e^{\eps|w-w_0|}\ \d\meas\leqslant 3,
		\end{equation}
		where 
		\begin{equation}
		w_0:=\intav_{B_{2r}(y)}w\ \d\meas.
		\end{equation}
		Thus 
		\begin{equation}
		\begin{aligned}
		\intav_{B_{2r}(y)}u^{-\eps}\ \d\meas\intav_{B_{2r}(y)} u^\eps\ \d\meas
		=&\intav_{B_{2r}(y)}e^{\eps w}\ \d\meas\intav_{B_{2r}(y)}e^{-\eps w}\ \d\meas\\
		=&\intav_{B_{2r}(y)}e^{\eps(w-w_{0})}\ \d\meas\intav_{B_{2r}(y)}e^{\eps(w_{0}-w)}\ \d\meas\\
		\leqslant &\left(\intav_{B_{2r}(y)}e^{\eps|w-w_{0}|}\ \d\meas\right)^2\leqslant 9.
		\end{aligned}
		\end{equation}
		By \eqref{eq:102},
		\begin{equation}
		\essinf{B_r(y)}u\geqslant C(N,s)\left(\intav_{B_{2r}(y)}u^\eps\ \d\meas\right)^\frac 1 \eps. 
		\end{equation}
		Therefore, recalling our choice of $\eps$ and \eqref{eq:103}, we have the desired inequality. 
		
		Finally let us assume $u$ is not bounded away from $0$. We can consider $u_\delta=u+\delta$ $(\delta>0)$ instead, and then let $\delta\downarrow 0$.
	\end{proof}

	We are now in a position to introduce a regularity result on ($\mathscr L$-)superharmonic functions. See \cite[Subsection 8.5]{BjornBjorn}, in particular, Proposition 8.24 therein.
	\begin{proposition}\label{lem:usc}
		Let $u$ be as in Proposition \ref{lem:weak.harnack.sub22}. Assume that $u$ is locally bounded. 
		Then there exists a unique representative $\bar u$ of $u$ such that 
		every $z\in B_{100r}(y)$ is a Lebesgue point of $\bar u$.
		Moreover $\bar u$ is upper semi-continuous satisfying that  
		for any $z\in B_{100r}(y)$
			\begin{equation}\label{eq:42}
			\bar u(z)=\limsup_{w\ra z} \bar u(w).
			\end{equation}
	Indeed, such a representative can be realized by
	\begin{equation}\label{loerrep}
	\bar u(z):=\lim_{\rho\ra 0}\esssup{w\in B_\rho(z)}u(w).
	\end{equation}
	\end{proposition}
	
	
	
	\begin{proof}
		Let $\bar u$ be as in (\ref{loerrep}).
	Firstly, since it is easily checked that the set $\{z\mathop| \bar u(z)<a\}$ is open for any $a \in \mathbb{R}$, $\bar u$ is upper semicontinuous. 
		
		Secondly, let us check $\bar u= u$ $\meas$-a.e. 
		Take a Lebesgue point $z\in B_{100r}(y)$ of $u$. Since $u$ is locally bounded, we have
		\begin{equation}
		\lim_{\rho\ra 0}\intav_{B_\rho(z)}\left|u-u(z)\right|^q\ \d\meas =0, \quad \forall q>0.
		\end{equation}
		For any $\eps>0$, there exists $r_0>0$ such that for all $0<\rho<r_0$, $B_\rho(z)\subset B_{100r}(y)$ and
		\begin{equation}
		\left(\intav_{B_\rho(z)}\left|u-u(z)\right|^p\ \d\meas\right)^{1/p}<\eps,
		\end{equation}
		where $p$ is as in Proposition \ref{lem:weak.harnack.sub22}.
		Applying Proposition \ref{lem:weak.harnack.sub22} proves
		\begin{equation}
		\bar u(z)-u(z)\leqslant\esssup{w\in B_{\frac \rho 2}(z)}\left(u(w)-u(z)\right)\leqslant C\left(\intav_{B_\rho(z)}\left|u-u(z)\right|^p\ \d\meas \right)^\frac 1 p\leqslant C\eps.
		\end{equation}
		Thus $\bar u(z)\leqslant u(z)$. On the other hand
		\begin{equation}
		\bar u(z)=\lim_{\rho\ra 0}\esssup{w\in B_\rho(z)}u(w)\geqslant \lim_{\rho\ra 0}\intav_{B_\rho(z)}u\ \d\meas=u(z).
		\end{equation}
		We obtain that $\bar u(z)=u(z)$. Thus the Lebesgue differentiation theorem allows us to conclude $\bar u=u$ for $\meas$-a.e. Moreover, observe that
		\begin{equation}
		\bar u(z)\geqslant \limsup_{w\ra z}\bar u(w)\geqslant \lim_{\rho\ra 0}\esssup{w\in B_\rho(z)} \bar u(w)=\bar u(z),
		\end{equation}
		therefore we obtain \eqref{eq:42}. 
		
		Thirdly, let us show that every point in the domain is a Lebesgue point of $\bar u$. Without loss of generality we may assume $0<\bar u<1$ in $B_{100r}(y)$ since $u$ is locally bounded. Fix any $z\in B_{100r}(y)$. For any $\eps>0$ with $u(z)<1-\eps$, by the upper semicontinuity, there exists $\rho>0$ sufficiently small such that $0<u(w)<u(z)+\eps<1$ for any $w\in B_\rho(z)$. Let $v:=-u+u(z)+\eps$, then $0<v<1$ in $B_\rho(z)$. Applying the weak Harnack inequality \eqref{eq:46} to $v$, there exist $p=p(N,s)>0$ and $C=C(N,s)>0$ such that 
		\begin{equation}\label{eq:49}
		\begin{aligned}
		\intav_{B_\rho(z)}|\bar u&-\bar u(z)|\ \d\meas=
		\intav_{B_\rho(z)} |v-v(z)|\ \d\meas\leqslant
		\eps+\intav_{B_\rho (z)}v^{1-p}v^p\ \d\meas\\&\leqslant
		\eps+\intav_{B_\rho(z)}v^p\ \d\meas\leqslant \eps+C\left(\inf_{B_\rho(z)}v\right)^p\leqslant\eps+C\eps^p.
		\end{aligned}
		\end{equation}
		Letting $\rho\ra 0$,  we conclude because $\eps$ is arbitrary.
		
		Finally, the uniqueness is obvious because every point in the domain is a Lebesgue point.
	\end{proof}
	
	
	In view of this, we always assume that any $\mathscr L$-subharmonic function $u$ is actually the canonical representative as obtained in Proposition \ref{lem:usc}. Therefore, for example, by \eqref{eq:42},  the ``$\mathrm{ess\ sup}$'' in Proposition \ref{lem:weak.harnack.sub22} (``$\mathrm{ess\ inf}$'' in Proposition \ref{prop:weak.harnack.super}, respectively) can be replaced by ``$\sup$'' (``$\inf$'', respectively). Under this convention, finally let us provide the strong maximum principle for $\mathscr L$-subharmonic functions.
	
	\begin{proposition}[Strong maximum principle for $\mathscr L$-subharmonic functions]\label{thm:strong.max.princi.}
		Let $u$ be a $\mathscr L$-subharmonic function on a connected open subset $\Omega$ in $X \setminus \{x\}$. If its supremum in $\Omega$ attains at a point in $\Omega$, then $u$ is constant.
	\end{proposition}
	\begin{proof}
		Denote by $A:=\sup_\Omega u$ and put $D:=\{x \in \Omega |u(x)=A\}$. It is trivial that $D$ is closed in $\Omega$. On the other hand, for any $y \in D$, 
		applying Proposition \ref{prop:weak.harnack.super} to $\bar u=A-u$ proves
		\begin{equation}
		0=\inf_{B_r(y)} \bar u\geqslant C\left(\intav_{B_{2r}(y)} \bar u^p\ \d\meas\right)^\frac 1 p,
		\end{equation}
		namely $\bar u=0$ $\meas$-a.e. in $B_{2r}(y)$ for any sufficiently small $r>0$, thus $\bar u \equiv A$ on $B_{2r}(y)$ because of the lower semicontinuity of $\bar u$. This shows that $D$ is open in $\Omega$, Thus $D=\Omega$ because $\Omega$ is connected.
	\end{proof}

\end{document}